\documentclass{aims}
\usepackage{amsmath}
  \usepackage{paralist}
  \usepackage{graphics} 
  \usepackage{epsfig} 
\usepackage{graphicx}  \usepackage{epstopdf}
 \usepackage[colorlinks=true]{hyperref}
\hypersetup{urlcolor=blue, citecolor=red}

\usepackage{tikz}

\def\currentvolume{X}
 \def\currentissue{X}
  \def\currentyear{200X}
   \def\currentmonth{XX}
    \def\ppages{X--XX}
     \def\DOI{10.3934/xx.xx.xx.xx}

\newtheorem{theorem}{Theorem}[section]
\newtheorem{corollary}{Corollary}

\newtheorem{lemma}[theorem]{Lemma}

\theoremstyle{definition}
\newtheorem{definition}[theorem]{Definition}
\newtheorem{remark}{Remark}

\newcommand{\p}{\partial}
\newcommand{\dd}{\mathrm{d}}
\newcommand{\norm}[1]{\left\Vert#1\right\Vert}

\newcommand{\R}{\mathbb{R}}
\newcommand{\D}{\mathbb{\D}}

\newcommand{\del}{{\partial}}
\renewcommand{\O}{{\mathbf O}}

\title[Stability of transonic jets] 
      {Stability of transonic jets with strong rarefaction waves for two-dimensional steady compressible Euler system}

\author[Min Ding and Hairong Yuan]{}

\subjclass{Primary: 35L50, 35L65, 35Q31, 35R35; Secondary: 76N10.}
 \keywords{Compressible Euler equations, characteristic discontinuity, transonic, rarefaction wave, wave front tracking, interaction of waves, Glimm functional.}

 \email{minding@whut.edu.cn}
 \email{hryuan@math.ecnu.edu.cn}

\thanks{The research of Min Ding is supported by the Fundamental Research Funds for the Central Universities (WUT: 2016IVA074) and by the National Natural Science Foundation of China under Grant Nos. 11626176 and 11701435.
The research of Hairong
Yuan  is supported by the National Natural Science Foundation of China under Grant No. 11371141, and by Science and Technology Commission of Shanghai Municipality (STCSM) under grant No. 13dz2260400.}

\thanks{$^*$ Corresponding author: Hairong Yuan}

\begin{document}
\maketitle

\centerline{\scshape Min Ding}
\medskip
{\footnotesize
 \centerline{Department of Mathematics, School of Science }
   \centerline{Wuhan University of Technology}
   \centerline{ Wuhan 430070, China}
} 

\medskip

\centerline{\scshape Hairong Yuan$^*$}
\medskip
{\footnotesize
 \centerline{Department of Mathematics, Center for PDE,}
 \centerline{and Shanghai Key Laboratory of PMMP}
   \centerline{East China Normal University}
   \centerline{Shanghai 200241, China}
}

\bigskip

 \centerline{(Communicated by the associate editor name)}

\begin{abstract}
We study supersonic flow past a convex corner which is surrounded
by quiescent gas. When the pressure of the upstream supersonic
flow is larger than that of the quiescent gas, there appears a
strong rarefaction wave to rarefy the supersonic gas. Meanwhile, a transonic characteristic discontinuity appears to separate the
supersonic flow behind the rarefaction wave from the static gas. In this paper, we employ a wave front tracking method to
establish structural stability of such a flow pattern under non-smooth
perturbations of the upcoming supersonic flow. It is an initial-value/free-boundary  problem for the two-dimensional steady non-isentropic
compressible Euler system. The main ingredients are careful
analysis of wave interactions and construction of suitable Glimm
functional, to overcome the difficulty that the strong rarefaction wave has a large total variation.
\end{abstract}

\tableofcontents
\section{Introduction}\label{S:1}
\begin{center}
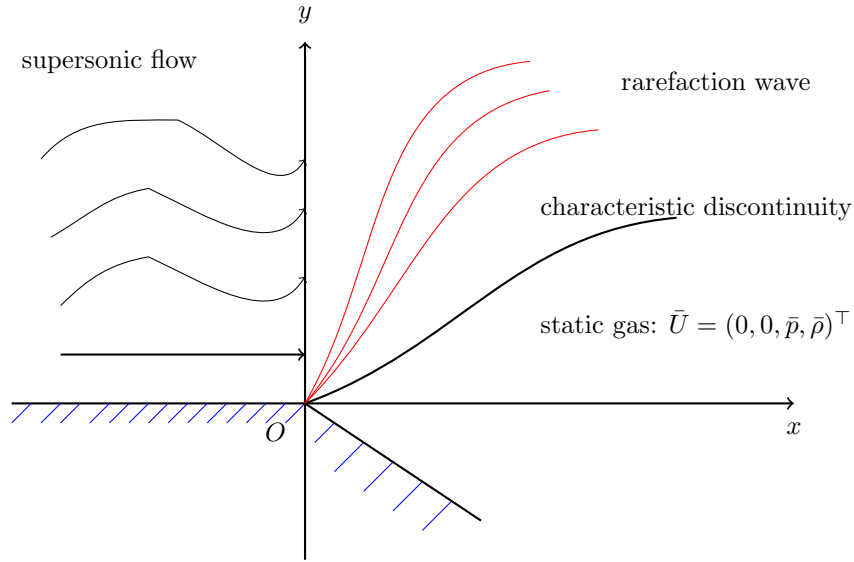
\begin{figure}
\begin{tikzpicture}[scale=1.3] \label{fig1}
\draw[thick][->](-3, 0)--(5, 0); \draw[thick][->](0, -1.6)--(0,
3.7);
 \draw [blue](-3, -0.2)--(-2.8,
0); \draw[blue](-2.7, -0.2)--(-2.5, 0); \draw[blue](-2.5,
-0.2)--(-2.3, 0); \draw[blue](-2.2, -0.2)--(-2, 0); \draw[blue](-2,
-0.2)--(-1.8, 0); \draw[blue](-1.8, -0.2)--(-1.6,0);
\draw[blue](-1.6, -0.2)--(-1.4, 0);
 \draw [blue](-1.4, -0.2)--(-1.2,
0); \draw[blue](-1.2, -0.2)--(-1, 0); \draw[blue](-1, -0.2)--(-0.8,
0); \draw[blue](-0.8, -0.2)--(-0.6, 0); \draw[blue](-0.6,
-0.2)--(-0.4, 0); \draw[blue](-0.4, -0.2)--(-0.2,0);
\draw[blue](-0.2,-0.2)--(0,0);
\draw [blue](0.1, -0.4)--(0.3, -0.2); \draw[blue](0.3, -0.7)--(0.6,
-0.4); \draw[blue](0.6, -0.9)--(0.9, -0.6); \draw[blue](0.9,
-1.1)--(1.2, -0.8); \draw[blue](1.2, -1.3)--(1.5, -1);
\draw[thick](0, 0)--(1.8,-1.2);
 \draw[thick][->](-2.5,  0.5)--(0, 0.5);
 \draw[thin][->](-2.5, 1) to
[out=45, in=-170](-1.6, 1.5) to[out=-25, in=-120](0, 1.3);
\draw[thin][->](-2.6, 1.7) to [out=30, in=-170](-1.6, 2.2)
to[out=-25, in=-120](0, 2); \draw[thin][->](-2.7, 2.5) to [out=50,
in=-180](-1.3, 2.9) to[out=-25, in=-120](0, 2.5);

\draw[thick](0, 0) to [out=20, in=-175](3.8, 1.9); \draw[red](0, 0)
to [out=45, in=-175](3, 2.8); \draw[red](0, 0) to [out=50,
in=-170](2.5, 3.2); \draw[red](0, 0) to [out=60, in=-175](2.3, 3.5);

\node at (4.0, 2.0) {$\text{characteristic discontinuity}$}; \node
at (4.2, 3.3) {$\text{rarefaction wave}$};
 \node[below] at (5,
-0.1) {$x$}; \node[below] at (-0.3, -0.1){$O$};
 \node [above]at (0, 3.8){$y$};
 \node at (4, 0.8){$\text{static gas: } \bar{U}=(0, 0, \bar{p},
 \bar{\rho})^\top$};
 \node at (-2, 3.5) {$\text{supersonic flow}$};
\end{tikzpicture}
\caption[]{\small  A transonic characteristic discontinuity separating supersonic flow behind the rarefaction wave and the surrounding static gas. }
\end{figure}
\end{center}

We are concerned with two-dimensional steady non-isentropic
compressible supersonic Euler flows passing a convex corner
surrounded by static gas. When the pressure of the upcoming
supersonic flow is larger than that of the quiescent gas, there may
appear a strong rarefaction wave to rarefy the upcoming flow to the
quiescent gas. Meanwhile, a characteristic discontinuity, which is a combination of vortex sheet and entropy wave, is generated to
separate the supersonic flow behind the rarefaction wave from the
still gas, see Figure 1. 
Under suitable assumptions on the upstream supersonic flow and the surrounding quiescent gas, we wish to establish structural stability of such a flow pattern in the class of functions with bounded variations  by considering an initial-value/free-boundary problem for the
two-dimensional steady full compressible Euler system.

Recall that the Euler system, consisting of conservation of mass, momentum and energy,
reads as
\begin{equation}\label{E:1.1}
\begin{cases}
\displaystyle \del_x(\rho u)
  + \del_y \left(\rho v\right)=0,\\[8pt]
\displaystyle \del_x(\rho u^2+p)
  + \del_y \left(\rho u v\right)=0,\\[8pt]
\displaystyle \del_x\left(\rho uv\right)
  + \del_y \left(\rho v^2+p \right) = 0,\\[8pt]
  \displaystyle \del_x\left(\rho u(E+\frac{p}{\rho})\right)
  + \del_y \left(\rho v(E+\frac{p}{\rho})\right) = 0,
\end{cases}
\end{equation}
where $\rho$,  $p$ and $(u, v)$ represent respectively  the density of mass, scalar pressure and velocity of the flow, and $E=\frac{1}{2}(u^2+v^2)+e$
is the total energy per unit mass, with $e$ the internal energy. Let
$S$ be the entropy. For polytropic gas, the constitutive
relations are given by
\[
p=\kappa \rho^{\gamma}\exp(S/c_v), \qquad
e=\frac{\kappa}{\gamma-1}\rho^{\gamma-1}\exp(S/c_v)=\frac{p}{(\gamma-1)\rho},
\]
with $\kappa, c_{v}, \gamma>1$ being positive constants. The corner and the Cartesian coordinates $(x,y)$ of the plane are illustrated in Figure 1. 

System \eqref{E:1.1} can be written in the general form of conservation laws:
\begin{equation}\label{E:1.2}
\del_x W(U)+\del_y H(U)=0, \quad U=(u, v, p, \rho)^{\top},
\end{equation}
where
\[
\begin{split}
&W(U)=\left(\rho u,\; \rho u^2+p, \;\rho uv,\; \rho u(E+\frac{p}{\rho})\right)^\top, \\[5pt]
&H(U)=\left(\rho v,\;\rho uv,\; \rho v^2+p,\; \rho
v(E+\frac{p}{\rho})\right)^\top.
\end{split}
\]
The eigenvalues of system \eqref{E:1.2}, namely, the roots of the
polynomial $\det(\lambda \nabla_U W(U)-\nabla_U H(U))$,  are
\[
\lambda_1=\frac{uv-c\sqrt{u^2+v^2-c^2}}{u^2-c^2}<
 \lambda_{2}=\frac{v}{u}< \lambda_3=\frac{uv+ c\sqrt{u^2+v^2-c^2}}{u^2-c^2},
\]
provided that $u>c>0$, where $\displaystyle c=\sqrt{{\gamma p}/{\rho}}$ is the local sonic
speed. The corresponding right eigenvectors
(null vectors of the matrix $\lambda\nabla_U W(U)-\nabla_U H(U)$)
are
\begin{eqnarray}\label{eq13}
\begin{split}
&r_j=k_j(-\lambda_j, 1, \; \rho(\lambda_j u-v), \frac{\rho(\lambda_j
u-v)}{c^2})^{\top},\; j=1, 3; \\ &r_{21}=(u, v, 0, 0)^{\top},\; r_{22}=(0, 0,0, \rho)^{\top}.
\end{split}
\end{eqnarray}
So for supersonic flow with $u>c>0$, all the characteristics are real and the four eigenvectors are linearly independent.  Hence by definition, the system \eqref{E:1.1} is hyperbolic in the positive $x$-direction if $u>c>0$, with constant characteristic multiplicities.
We may introduce the constants $k_j>0$ to normalize $r_j$ so that $\nabla_U \lambda_j(U)\cdot r_j(U)\equiv1$ for  $j=1, 3$ (see the Appendix for detailed computations), which means the first and the third characteristic families are genuinely  nonlinear. The second characteristic family is linearly degenerate, namely $\nabla_U \lambda_2(U)\cdot r_{2k}(U)\equiv0$  for $k=1,2.$

Next we formulate the domain and  boundary conditions to describe the transonic jet problem. Since we consider only time-independent flow, it is reasonable to assume that the state  below the
characteristic discontinuity is always static, and unchanged even if  the upcoming supersonic flow experienced  perturbations.  The static state is  denoted by $\bar{U}=(0, 0, \bar{p}, \bar{\rho})^{\top}$.  The characteristic discontinuity $\mathcal{C}$ itself is a free-boundary, which is the graph of a Lipschitz continuous function $y=g(x)$ on $x\geq 0$,
with $g(0)=0$. So the domain we consider is
$$\Omega_g=\{(x,y)\in\mathbb{R}^2: x>0, y>g(x)\}.$$
From the Rankine-Hugoniot jump conditions across a
characteristic discontinuity, we have the following two boundary conditions
\[
p=\bar{p}, \quad g'(x)=\frac{v(x, g(x))}{u(x, g(x))}\quad \text{ on }\;
y=g(x),
\]
where $\bar{p}$ is the constant pressure of the quiescent
gas. The second condition implies that $\mathcal{C}$ is
a characteristic boundary for the Euler equations \cite[p.3]{CKY1}. Regarding $x$ as time,  suppose that the supersonic flow on
$\mathcal{I}=\{(x, y)\in \R^2: x=0, y>0\}$  is given by
\[
U(0, y)=U_0(y), \qquad \text{with}\quad u_0(y)>c_0(y),
\]
which is the initial data.
Hence the main problem we will study in this paper is the Euler system \eqref{E:1.1} in the unknown domain $\Omega_g$, subjected to the following initial-value/free-boundary conditions
\begin{equation}
\label{E:1.3}
\begin{cases}
(u, v,  p, \rho)(0, y)=(u_0(y), v_0(y),  p_0(y), \rho_0(y)),&\text{ for}\; y\ge0;\\
p=\bar{p}, \qquad g'(x)=\frac{v}{u}(x, g(x)),&\text{ on }\;
y=g(x).
\end{cases}
\end{equation}

We will seek global entropy solutions of this problem.

\begin{definition}\label{D:1.1}
A pair of functions $(U(x,y), g(x))$, with $y=g(x)\in \mathrm{Lip}([0, +\infty); \mathbb{R})$,
and $U(x, y)\in L^\infty(\Omega_g;\mathbb{R}^4)$,  is an entropy solution of problem \eqref{E:1.2}\eqref{E:1.3}, if
\par \rm (i)  For any $\varphi(x, y)\in C_c^\infty(\mathbb{R}^2)$,  it holds that
\begin{eqnarray}
\begin{split}
&&\int_{\Omega_g}(W(U)\varphi_x+H(U)\varphi_y)\,\dd x \dd y+
\int_{\mathcal{C}}\left(H(U) -W(U)g'(x)\right)\varphi\,\dd \ell\\[5pt]
&&\ \ \ \  \  +\int^{+\infty}_{0}W(U_0)\varphi(0, y)\,\dd y=0;
\end{split}
\end{eqnarray}

\par \rm(ii) The following entropy inequality holds in the sense of
distributions:
\begin{equation*}
 \partial_x\left(\rho u S\right)+\partial_y \left(\rho v S\right)\geq 0 \quad \text{in} \;\;
\mathcal{D}'(\R^2).
\end{equation*}
\end{definition}

In the next section, we will show that for suitably chosen constant supersonic state $U_0(y)=U_+$ and static gas
$\bar{U}$, there exists an entropy solution $(U_\mathrm{b}, g_{\mathrm{b}})$ to
problem \eqref{E:1.2}\eqref{E:1.3}, which consists of a strong rarefaction wave of the
third characteristic family to decrease the pressure, and a transonic characteristic
discontinuity $y=g_\mathrm{b}(x)$ of the second characteristic family to separate the constant
supersonic state $U_-$ behind the rarefaction wave from the static gas $\bar{U}$.
Such a special solution is called a {\it background solution} in the sequel.
The purpose of this work is to show structural stability of these background solutions, which is described
in the following theorem.

\begin{theorem}\label{T:1.1}
For given constant supersonic state $U_+=(u_+, 0, p_+, \rho_+)^{\top}$, there is a positive number $p_*$. Let $\bar{p}$ be the pressure  of the static gas $\bar{U}$, with $p_*<\bar{p}<p_+$. Then
there exists a background solution $(U_{\mathrm{b}}, g_\mathrm{b})$
to problem \eqref{E:1.2}\eqref{E:1.3}, such that $u_\mathrm{b}>c_\mathrm{b}>0$.

Furthermore, there exist positive constants $\varepsilon_0, M_0, M_1$ such that if the initial data $U_0$ satisfies
\[
\|U_0(\cdot)-U_+\|_{L^{1}([0,+\infty))}+\mathrm{TV. }
U_0(\cdot)<\varepsilon\le\varepsilon_0,
 \]
where $\mathrm{TV. } U_0(\cdot)$ denotes the total variation of the
vector-valued function $U_0(y)$,  then problem
\eqref{E:1.2}\eqref{E:1.3} admits an entropy solution $(U(x,y), g(x))$, containing a 3-strong rarefaction wave, which is a small
perturbation of the background solution, in the sense that for almost all $x\ge0$, it holds
that
\begin{align}
|g'(x)-g'_\mathrm{b}(x)|\le& M_0\varepsilon,\label{freeest}\\
|\mathrm{TV. }\{p(x,\cdot):[g(x),+\infty)\}-\mathrm{TV. }\{p_{\mathrm{b}}(x,\cdot):[g_\mathrm{b}(x),+\infty)\}|\leq&
M_1\varepsilon,\label{eq:2.1.12}
\end{align}
and  for a.e. $(x,y)\in \Omega_g$,
\begin{equation}\label{E:1.5}
U(x, y)\in  D(U_+, \delta),
\end{equation}
with $\delta=M_1\varepsilon$.
The set $D(U_+, \delta)$ is defined by \eqref{E:3.4} in \S \ref{sec2}, which is a neighborhood of the background 3-rarefaction wave curve in the state space $\{U\in\mathbb{R}^4\}$, and is an invariant region for this problem.
\end{theorem}

We now review some known results on the stability of large solutions
for the stationary compressible Euler equations. The transonic characteristic discontinuity
separating supersonic flow from static gas  was firstly studied in
\cite{CKY1,CKY2013}. The case that the jet contains a strong shock
was solved  in \cite{KYZ}.  This paper is a continuation of these
works. The new feature here is the appearance of strong rarefaction
waves. To our knowledge, Zhang firstly studied an initial-boundary
value problem for the isentropic irrotational Euler equation with
strong rarefaction waves in \cite{Zhang}. Later, in \cite{CKZ1}, the
authors studied strong rarefaction wave in steady exothermically
reacting Euler flows,  and in  \cite{DKZ}, the authors considered a
piston problem for the case that  the piston was drawn away from the
gas which filled a straight thin tube ahead of the piston. There
appears a strong rarefaction wave in the tube. See also \cite{D1,D2,DL} for some generalizations. There are also many
works on flow fields containing strong supersonic shocks. Except
\cite{KYZ}, one may consult, for instance, the work of
Wang-Zhang \cite{WZ} on steady supersonic flow past a curved cone,
and Chen et.al. \cite{CZZ} on supersonic flow past a wedge.
Considerable progress has been made on the existence and
stability of transonic shocks in steady full Euler flows (see, for
example, \cite{LXY-2016} and references therein), which, unlike our case, are treating non-characteristic free-boundaries.  Structural
stability of supersonic characteristic discontinuities over Lipschitz
walls were thoroughly investigated by Chen et.al. in
\cite{CZZ-2006}.  There are also some results on transonic characteristic discontinuities for three-dimensional steady compressible Euler system, see \cite{WY2015,QX2017}.

In this paper we will employ a wave front tracking algorithm to construct a family of approximate solutions
to problem \eqref{E:1.2}\eqref{E:1.3}. Then by modifying appropriately the Glimm functional introduced
in \cite{GJ,bressan} and showing its monotonicity, we obtain compactness in the space of functions with bounded variations so that a subsequence of the approximate solutions converges to an entropy solution.  One of the main difficulties here is to obtain rather accurate estimates
when weak waves interact with strong rarefaction waves. One needs to introduce appropriate interaction potentials to
take into account of the fact that the total variation of the strong rarefaction wave is not small, and choose carefully weights for various terms in the Glimm functional. Some of the ideas were inspired by
\cite{amadori,Zhang}. However, different from previous works with strong rarefaction waves,  in our model a characteristic free-boundary, namely the strong transonic
characteristic discontinuity occurs. Therefore, the analysis is somewhat subtle and different. We recommend \cite{bressan, HR2015} for a general introduction of the front tracking method.

This paper is organized as follows. In \S \ref{sec2}, we review some
basic properties of elementary waves of the steady Euler system. Then we establish existence of
background solutions and show solvability of free-boundary
Riemann problems for system \eqref{E:1.1}. In \S \ref{E:4}, we outline the construction of
approximate solutions by wave front tracking method. In \S \ref{S:5}, we analyze local interaction estimates of perturbed waves and
reflections of waves on the transonic characteristic discontinuity in the front tracking process. Then we
construct a Glimm functional and prove its monotonicity. This manifests that the approximate solutions can really be computed for all $x>0$.
In \S \ref{sec5}, we show consistency, namely the limit of the approximate solution is actually an accurate entropy solution.  \S \ref{sec6} is a short appendix, contains some quantities used in the main text.

\section{Riemann Problems and Background Solutions}\label{sec2}

In this section, we find a background solution to problem
\eqref{E:1.2}\eqref{E:1.3} with constant initial data, and study reflections of waves from the characteristic free-boundary.

\subsection{Wave curves in state space}\label{sec21}
Firstly we consider Riemann problem of \eqref{E:1.2} with initial data:
\begin{equation}\label{E:2.1}
U|_{x=x_0}\doteq(u, v,  p, \rho)^{\top}|_{x=x_0}=
\begin{cases}U_L, \quad y<y_0,\\
             U_R, \quad y>y_0,
\end{cases}
\end{equation}
where $U_L=(u_L, v_L, p_L, \rho_L)^{\top}$ and $U_R=(u_R, v_R, p_R,
\rho_R)^{\top}$ represent the left (lower) and right (upper) states,
respectively. Solvability of the Riemann problem for general strictly
hyperbolic conservation laws with genuinely nonlinear or linearly degenerate characteristics  can be found in  \cite{bressan,HR2015,SJ,Da} when
$|U_L-U_R|$ is sufficiently small. The basic idea is to introduce
several wave curves in the state space, which can also be
used to construct large solutions, as shown below.

For any given left (resp. right) state $U_l$ (resp. $U_r$),  the set
of all possible states $U$ which can be connected to $U_l$ (resp.
$U_r$) on the right (resp. left) by $1$- or $3$-shock wave, is
denoted by $S_1(U_l)$ or $S_3(U_l)$ (resp. $S^{-1}_1(U_r)$ or $S^{-1}_3(U_r)$).
Similarly, we denote $R_1(U_l)$ or $R_3(U_l)$ (resp. $R^{-1}_1(U_r)$ or $R^{-1}_3(U_r)$)
the (inverse) wave curves of $1$- or $3$-rarefaction wave. We can
parameterize $R_j(U_l) \ (j=1,3)$ by
\[
\begin{cases}
\begin{split}
&\frac{\dd u}{\dd\alpha}=-k_j\lambda_j,\\
&\frac{\dd v}{\dd\alpha}=k_j,\\
&\frac{\dd p}{\dd\alpha}=k_j\rho(\lambda_j u-v),\\
&\frac{\dd \rho}{\dd\alpha}=k_j\frac{\rho(\lambda_j u-v)}{c^2},\qquad \alpha\ge0.
\end{split}
\end{cases}
\]
For our use, the $1$-inverse rarefaction wave curve $R^{-1}_1(U_r)$ is
\begin{equation}\label{eq22}
I(q, B)-\theta=I(q_r, B_r)-\theta_r,\quad q^2+\frac{2}{\gamma-1}
c^2=q^2_r+\frac{2}{\gamma-1} c^2_r,\quad
p\rho^{-\gamma}=p_r\rho^{-\gamma}_r,
\end{equation}
where $q\doteq\sqrt{u^2+v^2}$, $B\doteq q^2+\frac{2}{\gamma-1} c^2$,
$\theta\doteq\arctan \frac{v}{u}$, and
\[
\displaystyle I(\bar{q}, B)\doteq\int^{\bar{q}} \frac{\sqrt{q^2-c^2}}{q c} \text{d}q=\int^{\bar{q}}
\frac{\sqrt{(\gamma+1)t^2-(\gamma-1)B}}{t\sqrt{(\gamma-1)(B-t^2)}}
\dd t.
\]
Similarly, the inverse wave curve of 3-rarefaction wave $R^{-1}_3(U_r)$ is given by
\begin{equation}\label{eq23}
I(q, B)+\theta=I(q_r, B_r)+\theta_r,\;\; q^2+\frac{2}{\gamma-1}
c^2=q^2_r+\frac{2}{\gamma-1} c^2_r,\;\;
p\rho^{-\gamma}=p_r\rho^{-\gamma}_r.
\end{equation}
We remark that the above representation of the rarefaction wave curve is not complete. It only contains the part where $u>0.$

The second characteristic field is linearly degenerate. So it only supports characteristic discontinuities.  The corresponding wave curve passing through $U_l$ for vortex sheet (i.e. integral curve of vector field $r_{21}$ in state space $\mathbb{R}^4$)  is \footnote{Note that $e$ appeared below and in the rest of the paper is the base of the natural logarithm function, rather than the internal energy of the gas.}
\[
C_{21}(U_l): u=u_le^{\alpha_{21}}, \; v=v_l e^{\alpha_{21}}, \; p=p_l,\;
\rho=\rho_l,\quad\alpha_{21}\in\mathbb{R};
\]
and the  entropy wave curve (integral curve of $r_{22}$) through $U_l$ is given by
\[
C_{22}(U_l): u=u_l,\; v=v_l,\; p=p_l, \rho=\rho_l
e^{\alpha_{22}},\quad\alpha_{22}\in\mathbb{R}.
\]
The two waves coincide in the physical $(x,y)$-plane, with the same line $\{(x,y): y-y_0=(v_l/u_l)(x-x_0), x>x_0\}$ being their fronts.

The Rankine-Hugoniot jump conditions across the $j$-shock give the wave curves $S^{-1}_j(U_r)$:
\begin{eqnarray}
&&\left[p\right]=\frac{c^2_r}{b}\left[\rho\right],\;\;
[u]=-s_j[v],\quad \rho_r(s_j u_r-v_r)[v]=[p]  \;\;\text{ for } u>c,
j=1, 3,\label{E:2.3*}
\end{eqnarray}
where $[h]=h_r-h$ stands for the jump of a quantity $h$ across a
shock front; $s_j$ is the speed of the shock front:
\begin{eqnarray*}
s_j=\frac{u_rv_r+(-1)^{\sigma(j)}\bar{c}\sqrt{u^2_r+v^2_r-\bar{c}^2}}{u^2_r-\bar{c}^2},\quad
\text{with}\quad \bar{c}^2=\frac{\rho c^2_r}{\rho_r b},\quad
b=\frac{\gamma+1}{2}-\frac{\gamma-1}{2}\frac{\rho}{\rho_r},
\end{eqnarray*}
and $\sigma(1)=1, \sigma(3)=0.$
The entropy inequality means the pressure increases when particles pass a shock front, namely $[p]>0$ for 1-shock and  $[p]<0$ for 3-shock (cf. \cite[p.276]{HR2015}).

A solution to the Riemann problem
\eqref{E:1.2}\eqref{E:2.1} is given by at most four constant
states connected by shocks, characteristic discontinuities, and/or rarefaction waves. Exactly speaking, there exist piecewise $C^2$ curves
$\alpha_j\mapsto\Phi_j(\alpha_j; U)$, $j=1, 3$, and a $C^2$ surface $\alpha_2\doteq(\alpha_{21}, \alpha_{22})\mapsto\Phi_{2}(\alpha_{2}; U)$ (the latter is the point with parameter $\alpha_{22}$ on $C_{22}(U_m)$, while $U_m$ is the point with parameter $\alpha_{21}$ on curve $C_{21}(U)$), such that
\begin{equation}\label{E:4.1}
 \Phi(\alpha_1, \alpha_2, \alpha_3 ;U_L)\doteq\Phi_3\Big(\alpha_3;\Phi_{2}\Big(\alpha_2;
 \Phi_1\big(\alpha_1; U_L)\big)\Big)\Big)=U_R,
\end{equation}
whenever $|U_L-U_R|\ll 1$.

\begin{remark}\label{rm21}
Hereinafter, we denote by $\alpha_i, \beta_i, \gamma_i$ etc. the parameters for the $i$-wave curve/surface, $i=1, 2, 3$, while by their absolute
values the corresponding strengths of the waves.
It should be noted that since the Euler system is not strictly hyperbolic, we introduce the convention that strength of the characteristic discontinuity with parameter $\alpha_2$ is $|\alpha_2|\doteq|\alpha_{21}|+|\alpha_{22}|.$ We also use the parameters to represent the $i$-waves.
\end{remark}

To study reflection of waves on free-boundary, we introduce the notation $U_l=\Psi(\alpha_1, \alpha_2, \alpha_3;U_r)$ to represent the inverse wave curves; that means, the left state $U_l$ and the right state $U_r$ can be connected by 1-wave $\alpha_1$, 2-wave $\alpha_2$ and 3-wave $\alpha_3$. From the above constructions, we have
\begin{eqnarray}\label{E:2.3}
&&\frac{\partial \Psi}{\partial \alpha_i}(\alpha_1, \alpha_2,
\alpha_3;U_r )\big|_{\alpha_1=\alpha_2=\alpha_3=0}=-r_i(U_r),
\qquad i=1,3,\\
&&\frac{\partial \Psi}{\partial \alpha_{2k}}(\alpha_1, \alpha_2,
\alpha_3;U_r )\big|_{\alpha_1=\alpha_2=\alpha_3=0}=-r_{2k}(U_r),
\qquad k=1,2,
\end{eqnarray}
where $\alpha_2=(\alpha_{21}, \alpha_{22})$, and $r_i$, $r_{2k}$ are the right eigenvectors of the system
\eqref{E:1.1}, given by \eqref{eq13}. So $\alpha_i>0$ along
$R^{-1}_i(U_r)$, while $\alpha_i<0$ along $S^{-1}_i(U_r)$.
Particularly, we can parameterize
the 3-inverse rarefaction wave curve $R^{-1}_3(U)$ by solving
\begin{eqnarray}
&&\frac{d \Psi( 0, 0, \sigma;U)}{d\sigma}=-r_{3}(\Psi(0, 0,
\sigma;U)), \label{eq:2.8} \\[3pt]  &&\Psi( 0, 0, 0;U)=U.\label{eq:2.9}
\end{eqnarray}

\subsection{Background solutions}\label{sec22}
We now consider problem \eqref{E:1.2}\eqref{E:1.3}
with uniform upcoming supersonic flow whose pressure is quite large.

\begin{lemma}\label{lem1}
For given constant state $U_{+}$ satisfying $u_{+}>c_{+}$, there is a positive number $p_*$, such that if the
pressure $p_{+}$ is larger than the pressure $\bar{p}$ of the surrounding
static gas, and $\bar{p}>p_*$, then there exists a unique piecewise Lipschitz continuous
solution $(U_{\mathrm{b}}(x, y), g_\mathrm{b}(x))$ of problem
\eqref{E:1.1}\eqref{E:1.3}, with $U_0(y)\equiv U_+$, satisfying  $g_{\rm b}(x)=k_b x$, $u_{\mathrm{b}}(x, y)>c_{\mathrm{b}}(x, y)>0$, and for
$x>0$,
\begin{eqnarray}\label{eq25}
U_{\mathrm{b}}(x,y)=\begin{cases}
U_+,& y\ge k_1x,\\
U_{\rm ba}(\frac{y}{x}), & k_2 x<y<k_1x, \\
U_-, &k_bx<y\le k_2 x.
\end{cases}
\end{eqnarray}
Here $k_b, k_1, k_2$ are constants, while $U_{\rm ba}(\frac{y}{x})$ is a
3-rarefaction wave connecting the constant states $U_+$ and $U_-$,
with  $p_-=\bar{p}$. Furthermore, there exists a positive constant
$C$ such that
\begin{eqnarray}\label{eq26}
\mathrm{TV. }U_\mathrm{b}(x,\cdot)\le C|p_+-p_-|.
\end{eqnarray}
\end{lemma}

\begin{proof}
From the second and third equations in \eqref{eq23}, the entropy $S$
and Bernoulli constant $B$ are invariant across 3-rarefaction
waves, hence they are the same as those of $U_+$.

Set $M_1=u/c$. It is a function of $q$, namely,
$$M_1=\chi_1(q)\doteq\frac{q\cos\theta(q)}{c(q)},\quad \theta(q)=I(q_+,B_+)-I(q,B_+),\quad c(q)=\sqrt{\frac{2}{\gamma-1}(B_+^2-q^2)}.$$
By Bernoulli law, we also have
$$q=\chi_2(p)\doteq\left[B_+-\frac{\gamma(\gamma-1)}{2}A(S_+)^{\frac{1}{\gamma}}
p^{\frac{\gamma-1}{\gamma}}\right]^\frac12.$$
Hence $M_1=\chi(p)=\chi_1(\chi_2(p)).$ It is obvious that $\chi(p_+)=u_+/c_+>1$. By continuity, there is a number $p_*>0$ so that $M_1=\chi(p)>1$ for $p\in(p_*, p_+].$

Now by requiring
$p_-=\bar{p}>p_*$, and $p_-<p_+$, we determine $q_-=\chi_2(p_-)>q_+$. Using the first equation
in \eqref{eq23}:
$$I(q_+,B_+)=I(q_-,B_+)+\theta_-,$$
we solve $\theta_-$, which is negative, hence $k_b=\tan \theta_-$
and we find the downstream state $U_-$ in \eqref{eq25}.  Then
$k_1=\lambda_3(U_+)$ and $k_2=\lambda_3(U_-)$. For given $\theta\in
(\theta_-,0)$, we solve the speed $q_\sharp$ and the velocity $(u_\sharp,
v_\sharp)=(q_\sharp\cos\theta, q_\sharp\sin\theta)$ from
$$I(q_+,B_+)=I(q_\sharp,B_+)+\theta,$$
and hence determine the state $U_\sharp(\theta)$.
Then $U_{\rm ba}(y/x)=U_\sharp(\theta)$, where we could solve $\theta$ from
$\lambda_3(U_\sharp(\theta))=y/x$, by virtue of the genuine
nonlinearity of $\lambda_3$.  The estimate \eqref{eq26} follows
from monotonicity of $p, \rho, q$ and constancy of $S, B$ (see \eqref{eq22}) along the 3-rarefaction wave curve (cf. Lemma \ref{l3} below). Since the pressure is decreasing along rarefaction waves, one checks easily that $u_->c_-$ by our choice of $p_*$ and $p_-$.
\end{proof}

\begin{remark}
We observe that  no vacuum could appear. This is different from the
piston problem.
\end{remark}

For small $\delta_0>0$ (which is then fixed so that all the Riemann problems are solvable in the sequel), we introduce
the perturbation domain $D(U_+, \delta_0)$ in state space as follows:
\begin{eqnarray}\label{E:3.4}
D(U_+,\delta_{0})=\left\{U\in\mathbb{R}^4:
\begin{array}{l}
\Big|I(q, B)+\theta-I(q_+, B_+)\Big|<\delta_{0},\\ [8pt]
\displaystyle\big|q^2+\frac{2}{\gamma-1}c^2-q^2_+-\frac{2}{\gamma-1}c^2_+|<\delta_0,\\[8pt]
|p\rho^{-\gamma}-p_+\rho^{-\gamma}_+|<\delta_0,\\[5pt]
 \bar{p}-\delta_0<p<p_{+}+\delta_0
\end{array}
\right \},
\end{eqnarray}
which may be considered as a neighborhood in $\mathbb{R}^4$ of the 3-rarefaction wave curve passing through $U_+$, with pressure lying in $[\bar{p}, p_+]$.
The following lemma shows we could use the variation of pressure to measure the strength of large rarefaction waves in $D(U_+, \delta_0)$, just as in \eqref{eq26}. We use $V^{(k)}$ to denote the $k$-th argument of a vector $V$.

\begin{lemma}\label{l3}
There exists a $\delta'_*>0$ such that for any $U_l\in
D(U_+,\delta_{0})$, the function $\Phi_3^{(3)}(\alpha_3; U_l)$ is
strictly increasing with respect to $\alpha_3$ in  $\{\alpha_3 \big|
\alpha_3 \geq-\delta'_{*},\ \Phi_3(\alpha_3;U_l)\in
D(U_+,\delta_{0})\}$. Moreover, there exist two positive constants
$C_1$ and $C_2$ such that
\begin{eqnarray}\label{E:5.1}
C_{1}|\alpha_3|\leq| \Phi_{3}^{(3)}(\alpha_3; U_l)- p_l|\leq
C_{2}|\alpha_3|.
\end{eqnarray}
\end{lemma}

\begin{proof}
For $\alpha_3\ge0$,  by the properties of the rarefaction waves, we
have
\begin{eqnarray*}
\begin{split}
\frac{\dd
}{\dd\alpha_{3}}\Phi_{3}^{(3)}(\alpha_3; U_l)
&=\left.\Big(r_{3}\Big)^{(3)}\right|_{U=\Phi_{3}(\alpha_3; U_l)}=k_3\rho u(\lambda_3(U)-\frac{v}{u})|_{U=\Phi_3(\alpha_3; U_l)}>0\\[5pt]
\end{split}
\end{eqnarray*}
for any $U_l$ and $\Phi_{3}(\alpha_3; U_l)\in D(U_+,\delta_{0})$. So by continuous differentiability of the curve $\Phi_{3}(\alpha_3; U_l)$, there is a $\delta'_*>0$ so that for $-\delta'_*\leq \alpha_3<0$, we still know the derivative calculated above is positive.
The estimate \eqref{E:5.1} follows from the mean value theorem and inverse function theorem of differentiable  monotonic functions.
\end{proof}

The following lemma shows that strength of the rarefaction wave in the background solution is Lipschitz continuous with respect to $U_+$.

\begin{lemma}\label{lem23}
Let $U_r\in D(U_+,\delta_0)$.  Suppose that $$\Psi^{(3)}(0,0,\underline{S}; U_+)=\bar{p}=\Psi^{(3)}(0,0,S; U_r).$$ Then there is a constant $C$ depending only on $\delta_0$ so that
\begin{eqnarray}\label{eq213}
|S-\underline{S}|\le C|U_r-U_+|.
\end{eqnarray}
\end{lemma}

\begin{proof} Let $f(\alpha; U)\doteq\Psi^{(3)}(0,0,\alpha; U)$.
By the parametrization \eqref{eq:2.8}\eqref{eq:2.9}, thanks to continuous differentiability of solutions of ordinary differential equations with respect to parameters, we infer that $f$ is a $C^1$ function of $(\alpha, U)$, and
\[
\frac{\text{d}f(\alpha; U)}{\text{d}\alpha}=-r_3^{(3)}(\Psi(0,0,\alpha; U))<0.
\]
Thus, by the inverse function theorem of differentiable monotonic functions, there exists a $C^1$ function $h$ such that
\[
\alpha=h(p, U_r).
\]
Since $\bar{p}=f(\underline{S}; U_{+})$, $\bar{p}=f(S; U_r)$, then
\[
\underline{S}=h(\bar{p}, U_{+}), \qquad S=h(\bar{p}; U_r),
\]
and mean value theorem implies \eqref{eq213}.
\end{proof}
\begin{center}
\begin{figure}\label{fig2}
\begin{tikzpicture}[scale=1]
 \draw[black][->](0, 0)--(3, 2);
  \draw[black][->](3, 2)--(5, 4);
 \draw[thick][->](1, 3)--(3, 2);
 \draw[thick][->](3, 2)--(4, 4.2);
 \draw[dashed](3, -0.1)--(3, 4);
\node[above] at (0.8, 3){$\alpha_1$}; \node[above right] at (4,
4.2){$\beta_3$};
 \node at (1.5, 2) {$U_{l}$};
 \node at (2.5, 3.3)
{$U_{r}$}; \node at (4.1, 3.6){$U_m$}; \node [right]at (5.2,
4){$y=g(x)$};
\end{tikzpicture}
\caption[]{\small  Reflection on the free-boundary. }
\end{figure}
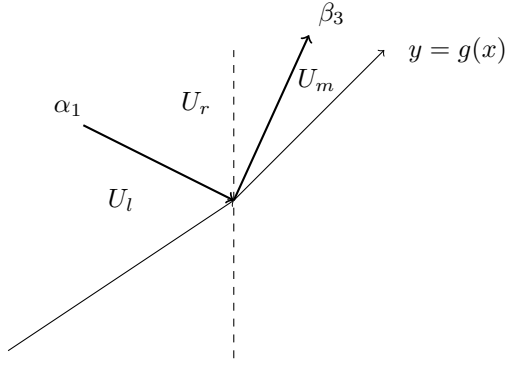
\end{center}
\subsection{Reflection of waves on the  characteristic free-boundary}\label{S:4.3}
Assume that a 1-wave front $\alpha_1$ hits the transonic free-boundary at
$(x_0,y_0)$ (cf. \S \ref{S:4.1} below that a rarefaction wave is replaced
by several fronts in the $(x,y)$-plane, and we need only consider one
front meets the free-boundary here, as shown in Figure \ref{fig2}). Let $U_l$ and $U_r$ be the left and right
states of $\alpha_1$, respectively. The reflected wave is a $3$-wave,
denoted by $\beta_3$.  Suppose the left and right states of $\beta_3$ are $U_{m}$ and $U_r$,
respectively.
Consider the free-boundary Riemann problem
\begin{equation}\label{eq:3.10}
\begin{cases}
\eqref{E:1.1}&\text{ in }x>x_0, \ \ y>k(x-x_0)+y_0,\\
U=U_r &\text{ on } x=x_0,\ \  y>y_0,\\
p=\bar{p} & \text{ on } x>x_0,\ \  y=k(x-x_0)+y_0,
\end{cases}
\end{equation}
where $k$ is a constant to be solved.

\begin{lemma}\label{l:4.3}
There exist positive constants $\varepsilon_0, M$, such that if $U_l, U_r\in B(U_-,\varepsilon)$ (the ball with center $U_-$ and radius $\varepsilon\le \varepsilon_0$ in $\mathbb{R}^4$), and  $U_l=\Psi(\alpha_1, 0,0;U_{r})$, then problem \eqref{eq:3.10} admits a
unique simple wave solution $(U_m, y=k(x-x_0)+y_0)$ for $x>x_0$, satisfying $U_m=\Psi(0,0,\beta_3;U_{r})$,  $k=v_m/u_m$, and
\begin{eqnarray}
\beta_3=-K_{b_1}\alpha_1,\quad |U_m-U_-|\le M\varepsilon,
\end{eqnarray}
where $K_{b_1}$ is positive and uniformly bounded,
with a bound $C_b$ depending only on the background solution and $\delta_0$.
\end{lemma}
We recall that $\delta_0$ appeared in \eqref{E:3.4}, which is small and fixed.

\begin{proof}
By the boundary conditions on the free-boundary,  we
have
\begin{equation}\label{E:4.1}
\Psi^{(3)}( 0, 0, \beta_3;U_r)=\Psi^{(3)}(\alpha_1, 0, 0;U_r)=\bar{p}.
\end{equation}
Differentiate \eqref{E:4.1} with respect to $\beta_3$,
and let $\beta_3=0$ (hence $\alpha_1=0$), we see
\[
\frac{\partial \Psi^{(3)}}{\partial \beta_3}|_{U=U_r,
\beta_3=0}=-(r_{3})^{(3)}(U_r)\neq 0.
\]
By the implicit function theorem, close to $\alpha_1=0$, there
exists a function $f_1\in C^1$ such that
\begin{equation}\label{E:4.2}
\beta_3=f_1(\alpha_1),\qquad f_1(0)=0.
\end{equation}
Hence $U_m$ is solved by
\[
U_m=\Psi(0, 0, \beta_3;U_r).
\]
Now consider the left-hand side of \eqref{E:4.1} as a function of $\alpha_1$, and differentiate it with respect to $\alpha_1$, we have
\[
\frac{\partial \Psi^{(3)}}{\partial \beta_3}\frac{\partial
\beta_3}{\partial \alpha_1}-\frac{\partial \Psi^{(3)}}{\partial
\alpha_1}=0.
\]
Let $\alpha_1=0$. Then
\[
\frac{\partial \beta_3}{\partial
\alpha_1}|_{U=U_{r},\alpha_1=0}=\frac{-r_{1}^{(3)}(U_{r})}{-r_{3}^{(3)}(U_{r})}=
\frac{k_1(\lambda_1(U)-\frac{v}{u})}{k_3(\lambda_3(U)-\frac{v}{u})}|_{U=U_{r}}.
\]
From Taylor's  formula, we have
\[
\beta_3=f_1(\alpha_1)-f_1(0)=-K_{b_1}\alpha_1,
\]
and
$$
K_{b_1}|_{\alpha_1=0, U_r=U_{-}}=-\frac{k_1(\lambda_1(U)-\frac{v}{u})}{k_3(\lambda_3(U)-\frac{v}{u})}|_{U=U_{-}}>0.
$$
The explicit expressions of $k_1$ and $k_3$ in \eqref{A1} were used here to determine the sign.
Since $K_{b_1}$ is continuous with respect to $\alpha_1$ and $U_r$, for $\varepsilon_0>0$ small,
if $|\alpha_1|<\varepsilon_0$ and $|U_r-U_-|<\varepsilon_0$, we still have
\[
\frac12K_{b_1}|_{\alpha_1=0,U_r=U_-}<K_{b_1}(\alpha_1, U_r)<\frac32K_{b_1}|_{\alpha_1=0,U_r=U_-}.
\]
Thus, $K_{b_1}$ is positive and bounded by a constant $C_b$ depending only on the background solution and $\delta_0$.

Since $\Psi$ is $C^1$, there are constants $C'_1, C_2$  such that
$$|U_m-U_r|\le C'_1|\beta_3|\le C_bC'_1|\alpha_1|<C_2|U_l-U_r|<2C_2\varepsilon.$$
So $|U_m-U_-|<M\varepsilon$, with $M=2C_2+1.$
\end{proof}

\begin{remark}\label{rm22}
This lemma says that  $1$-wave front is changed to $3$-wave
after reflection, and rarefaction front becomes shock front, and vice versa. It shows how to solve the free-boundary as ``times" evolves. We remark that $U_{m}\in B(U_-,M\varepsilon)$ implies that $|k-k_\mathrm{b}|\le M_0\varepsilon$, for a constant $M_0$ depending only on the background solution,  by noting that the velocity $u$ has a positive lower bound in $D(U_+,\delta_0)$, and using the mean value theorem for continuously differentiable functions. This fact finally leads to \eqref{freeest} claimed in Theorem \ref{T:1.1}.
\end{remark}

\section{Construction of  Approximate Solutions}
\label{E:4}

In this section, following the general ideas presented in \cite{bressan,HR2015}, we define accurate and simplified Riemann solvers, which are both appropriate modifications of the exact solutions of the Riemann problems indicated in \S \ref{sec2}. They are
building blocks to construct approximate solutions of problem
\eqref{E:1.2}\eqref{E:1.3} by a wave front tracking algorithm.

\subsection{Riemann solvers}\label{S:4.1}

The definitions given below are standard, see \cite[p.129]{bressan} or \cite[p.286]{HR2015}.

\vspace{0.2cm}
\underline{Case 1. Accurate Riemann solver}

\vspace{0.1cm}

Let $\delta>0$ be given, which measures accuracy of  each approximate solution constructed by front tracking later. The accurate Riemann solver is as mentioned in \S \ref{sec2}, except that every
rarefaction wave $R_i$, $i=1, 3$, with strength $\alpha_i>0$,  has been divided into $\nu$ equal
parts and replaced by $\nu$ rarefaction fronts. Here $\nu$ is the smallest positive integer larger than or equal to $\alpha_i/\delta$.

More precisely, suppose that the left state $U_L$ and the right state $U_R$ are connected by a $1$-rarefaction wave $\alpha_1$.  Let $U_{0,0}=U_{L}$, $U_{0,\nu}=U_{R}$, and for any $1\leq k\leq\nu$, set
\begin{eqnarray*}
U_{0, k}=\Phi_{1}\Big(\frac{1}{\nu}\alpha_{1}; U_{0, k-1}\Big),\quad
y_{1,k}=y_{0}+(x-x_{0})\lambda_{1}(U_{0, k}),\quad x>x_0.
\end{eqnarray*}
Then we replace the $1$-rarefaction wave by piecewise constant states for $x>x_0$:
\begin{eqnarray}\label{eq:3.2}
U^{\delta}_{A}(U_L,U_R)=\left\{
\begin{array}{lllll}
      U_{L}, \  &y<y_{1,1},\\[3pt]
      U_{0,k}, \  &y_{1,k}<y<y_{1,k+1},\quad k=1,\ldots,\nu-1, \\[3pt]
      U_{R},\  &y_{1,\nu}<y<y_{0}+(x-x_{0})\lambda^{*}_{1},
\end{array}
 \right.
\end{eqnarray}
where $\lambda^{*}_{1}\in(\max_{U\in D(U_+,\delta_0)}\lambda_{1}(U), \min_{U\in D(U_+,\delta_0)}\lambda_2(U))$ is a fixed number.
Similarly, we can approximate 3-rarefaction wave by
$\nu$ 3-rarefaction fronts  in the domain $\{(x, y): x>x_0,
y>y_0+\lambda^*_3(x-x_0)\}$, with $$\lambda^*_3\in(\max_{U\in D(U_+,\delta_0)} \lambda_2(U), \min_{U\in D(U_+,\delta_0)}
\lambda_3(U)).$$

\vspace{0.1cm}

\underline{Case 2. Simplified Riemann solver}

\vspace{0.1cm}

Let $\hat{\lambda}$ (strictly larger than all the eigenvalues of system \eqref{E:1.2} in $D(U_+,\delta_0)$)  be the speed of artificial discontinuities called non-physical fronts,
which are introduced so that the total number of wave fronts (discontinuities) is finite for all
$x\geq 0$ for a given approximate solution of problem \eqref{E:1.2}\eqref{E:1.3}.  The strength of a non-physical wave is the error due
to the following simplified Riemann solver. It occurs in three
cases:

\underline{Case a.} A $j$-wave $\beta_j$ and an $i$-wave $\alpha_i$
interact at $(x_0, y_0)$, $1\leq i\leq j\leq 3$. Suppose that $U_L$,
$U_M$ and $U_R$ are three constant states, satisfying
\begin{eqnarray}
U_{M}=\Phi_{j}(\beta_j; U_{L}),\ U_{R}=\Phi_{i}(\alpha_i; U_{M}).
\end{eqnarray}
We define an auxiliary right state
\begin{eqnarray}
U'_{R}=\left\{
\begin{array}{lll}
      \Phi_{j}(\beta_j;\Phi_{i}(\alpha_i; U_{L}) ), \ &j>i,\\[8pt]
      \Phi_{j}(\alpha_j+\beta_j; U_{L}), \ &j=i.
\end{array}
 \right.
\end{eqnarray}
The simplified Riemann solver $U_S(U_L, U_R)$ at $(x_0, y_0)$ of
problem \eqref{E:1.2}\eqref{E:2.1} is
\begin{eqnarray}
U_{S}(U_{L},  U_{R})=\left\{
\begin{array}{ll}
     U^{\delta}_{A}(U_{L},U'_{R} ), \ &y-y_{0}<\hat{\lambda}(x-x_{0}),\ \ x>x_0,\\[5pt]
      U_{R}, \ &y-y_{0}>\hat{\lambda}(x-x_{0}),\ \ x>x_0,
\end{array}
 \right.
\end{eqnarray}
where $\displaystyle  U^{\delta}_{A}(U_{L},U'_{R} ) $ is constructed by the accurate Riemann solver as in Case 1 (the rarefaction waves are split, while all  shocks and characteristic discontinuities remain unchanged).
The non-physical front is defined by
\begin{eqnarray}
U_{np}=\left\{
\begin{array}{ll}
     U'_{R}, \ &y-y_{0}<\hat{\lambda}(x-x_{0}),\ \ x>x_0,\\[5pt]
      U_{R}, \ &y-y_{0}>\hat{\lambda}(x-x_{0}),\ \ x>x_0,
\end{array}
 \right.
\end{eqnarray}
and the strength of the non-physical front is $\epsilon=|U_R-U'_R|$.

\vspace{0.2cm} \underline{Case b.} A non-physical front interacts
with an $i$-wave front $\alpha_i$ ($i=1, 2, 3$) coming from the
above/right at $(x_0, y_0)$. Suppose that the three states $U_L$, $U_M$
and $U_R$ satisfy
\begin{eqnarray*}
|U_{M}-U_{L}|=\epsilon,\quad  U_{R}=\Phi_{i}(\alpha_i; U_{M}).
\end{eqnarray*}
Then the simplified Riemann solver $U_{S}(U_{L},  U_{R})$ of problem \eqref{E:1.2}\eqref{E:2.1} is
\begin{eqnarray*}
U_{S}(U_{L}, U_{R})=\left\{
\begin{array}{lll}
    U_A^\delta(U_L, \Phi_{i}(\alpha_i; U_{L})),\ & y-y_{0}<\hat{\lambda}(x-x_{0}), \\[5pt]
      U_{R},\ &y-y_{0}>\hat{\lambda}(x-x_{0}),
\end{array}
 \right.
\end{eqnarray*}
where $x>x_0.$ This means that strength of the physical front is unchanged after interaction, namely still to be $\alpha_i$, while strength of the non-physical front becomes to be $\epsilon'=|\Phi_{i}(\alpha_i; U_{L})-U_{R}|.$ In particular, if $\alpha_i$ is a rarefaction front, there is no splitting of rarefaction waves in the solver $U_A^\delta(U_L, \Phi_{i}(\alpha_i; U_{L})),$ since $|\alpha_i|\le\delta.$

\vspace{0.2cm} \underline{Case c.} A 1-wave front $\alpha_1$ hits the free-boundary. In Lemma \ref{l:4.3} we have shown the accurate Riemann solver, where $\beta_3$ is the reflected wave, separating $U_m$ and the new free-boundary from the states $U_r$ behind $\alpha_1$.  For the simplified Riemann solver, we just replace $\beta_3$ by a non-physical front $\epsilon$ travelling with speed $\hat{\lambda}$. So the state ahead  and behind of $\epsilon$ is still $U_r$ and $U_m$ respectively, and the free-boundary is unchanged.
By Lemma \ref{l:4.3}, we  have the estimate
\begin{eqnarray}\label{eq36}
\epsilon\doteq|U_m-U_r|\le C'_1|\beta_3|=C'_1K_{b_1}|\alpha_1|\le C'_1 C_b|\alpha_1|.
\end{eqnarray}

The non-physical fronts will be considered as fronts of the fourth family in the rest of the paper.

\subsection{Approximate solutions}\label{sec32}

For any sufficiently small $\delta>0$,
we construct a $\delta$-approximate solution $(U^{\delta}(x, y),g^\delta(x))$ to \eqref{E:1.2}\eqref{E:1.3} by
induction in the region $\{x>0, y\in\mathbb{R}\}$ as follows: \vspace{0.2cm}

\begin{itemize}
\item[ Step $1$. ] For $x=0$, we approximate the initial data $U_0$ by $U^{\delta}(0, y)$ ($y\ge0$), a piecewise constant function with finite jumps. It is required that
    \[\mathrm{TV}. U^{\delta}(0, y)\le \varepsilon\le\varepsilon_0,\quad \norm{U^{\delta}(0, y)-U_{+}}_{L^1([0,+\infty))}<\delta.\]
     Note that our assumptions in Theorem \ref{T:1.1} imply that $\lim_{y\to+\infty}U_0(y)=U_+$. So $U^\delta(0,y)=U_+$ for large $y$. The number $\varepsilon_0$ (see Theorem \ref{T:1.1}) is chosen small so that all the standard Riemann problems at the discontinuous points of $U^{\delta}(0, y)$ are solvable.

    At the corner $(0,0)$, as in Lemma \ref{lem1}, we solve a free-boundary Riemann problem with a strong rarefaction wave.

    Then we
approximate all the rarefaction waves appeared in these Riemann problems by rarefaction fronts as
described by Case 1 in \S \ref{S:4.1}. \vspace{0.2cm}

\item[ Step $2$. ] By induction, we assume that $(U^{\delta}, g^\delta)$ has been constructed for $x<\tau$, for some
$\tau>0$, and assume that $U^{\delta}|_{x<\tau}$ consists of a finite
number of wave fronts and for the first time, some of them interact, at $x=\tau$.  As shown in \S \ref{S:4.1}, we solve
a Riemann problem when two wave fronts interact, or a free-boundary
Riemann problem when a wave front hits the free-boundary. Thus we extend the approximate solution $(U^\delta,g^\delta)$ beyond $x=\tau$.
\end{itemize}

\begin{remark}
We may adjust speeds of wave fronts by a quantity arbitrarily small, so that there are no more than two discontinuities intersect at a point $(\tau, y_0)$, and only one wave front hits the free-boundary for each ``time" $x>0$.  Also, at any ``time" $x=\tau$, only one interaction occurs \cite[p.132, Remark 7.1]{bressan}.  Otherwise, we need technically more complicate estimates of wave interactions as in \cite[p.290]{HR2015}. Also, by our choice of speed $\hat{\lambda}$, a non-physical front will never hit the free boundary.
\end{remark}

To distinguish those fronts obtained by splitting the strong rarefaction wave that issued from the corner and the other weak rarefaction fronts coming from perturbations of the initial date, we  assign to each front of an approximate solution an integer called {\it generation order} in the following way, cf. \cite[p.300]{HR2015}. This is also used as a tool to estimate the total strength of non-physical fronts, at any given ``time" $x$, of an approximate solution.

\begin{itemize}
\item[ (A)] All wave fronts issued from $x=0,y>0$  have order $1$; the free-boundary, as well as  all the rarefaction fronts obtained from splitting the rarefaction waves issued from the corner $(0,0)$, have order $0$.

\item[(B)] A wave front of order $k$ hits the free-boundary, then  the generation order of the reflected fronts are still $k$.

\item[(C)] An $i$-wave front $\alpha_i$ of order $k_1$ interacts with a $j$-wave front $\beta_j$ of order $k_2$ at a point $(\tau, y_0)$. We then assign  generation order to the produced $l$-wave fronts $(l=1,2,3,4)$ to be
    \begin{equation}\label{E:4.4}
\begin{cases}
k_1+k_2 &\text{ if } l\neq i, j,\\
\min\{ k_1, k_2\} &  \text{ if }l=i=j,\\
k_1 &\text{ if }l=i\neq j,\\
k_2 &  \text{ if } l=j\neq i.
\end{cases}
\end{equation}
\end{itemize}

\begin{definition}(Strong rarefaction front)
A front $s$ is called a strong rarefaction front
provided that $s$ is a 3-rarefaction wave front with generation order $0$.
Otherwise, it is called a weak front.
\end{definition}

There may appear weak front of generation order zero, but all non-physical fronts have order at least one.

We now indicate what solver is used to solve Riemann problems encountered in Step 2 above. To be more specific,
there is at most one of the following  six cases occurring at the interaction ``time" $x=\tau$:

\begin{itemize}
\item[] \textbf{Case 1.} Two weak fronts $\alpha_{i}$ and $\beta_j$ $(1\leq i,j\leq3)$ interact;

\vspace{0.1cm}

\item[] \textbf{Case 2.} A 1-weak front $\alpha_1$ hits the characteristic free-boundary;

\vspace{0.1cm}
\item[] \textbf{Case 3.} A strong rarefaction front $s$
interacts with an $i$-weak wave front $\alpha_i$ $(i=1,2)$ from the above at ``time" $x=\tau$;

\vspace{0.1cm}
\item[] \textbf{Case 4.} A strong rarefaction front $s$
interacts with a 3-weak shock front $\alpha_3$ from the below (above)  at ``time" $x=\tau$;

\vspace{0.1cm}
\item[] \textbf{Case 5.} A strong rarefaction front $s$
interacts with a non-physical front $\epsilon$;

\vspace{0.1cm}
\item[] \textbf{Case 6.} A non-physical front $\epsilon$ interacts with an
$i$-weak wave $\alpha_i$ from the above ($i=1, 2, 3$).
\end{itemize}
For the given approximate solution $(U^\delta, g^\delta)$,  set
\begin{eqnarray}\label{E:5.3}
E_{\delta}(\tau-)=\left\{
\begin{array}{llll}
|\alpha_i||\beta_j|, & \quad \text{case 1}, \quad i,j\in\{1, 2, 3\},\\[5pt]
|\alpha_1|,& \quad \text{case 2},\\[5pt]
|\alpha_i||s|, & \quad \text{case 3}, \quad i=1, 2, \\[5pt]
\min\{|\alpha_3|, |s|\},   & \quad \text{case 4},\\[5pt]
|s||\epsilon|, & \quad \text{case 5},\\[5pt]
|\alpha_i||\epsilon|, &\quad \text{case 6},\quad i=1, 2, 3.
     \end{array}
     \right.
\end{eqnarray}

\begin{itemize}
\item[\underline{Rule 1.}] For Case 1--Case 4,  if $E_\delta(\tau-)>\mu_{\delta}$,  then the Riemann problem at the interacting point is solved by accurate Riemann solver. Otherwise we adopt the simplified Riemann solver. Here $\mu_{\delta}$ is a small parameter depending on $\delta$, which is to be specified  in \eqref{eq55}.

\vspace{0.3cm}
\item[\underline{Rule 2.}] For Case 5--Case 6,  we always use the simplified Riemann solver.

\end{itemize}

\begin{remark}
To make sure we could construct the approximate solution $(U^\delta, g^\delta)$ in $\{0\le x<T\}$ for any given large $T$, by the above front tracking algorithm, we need to guarantee the following:
\begin{itemize}
\item[(A).] For any $\tau>0$, $U^\delta(\tau)\in D(U_+,\delta_0)$, and the perturbation of total variation is small, so each Riemann problem could be solved;

\item[(B).] At any ``time" $x=\tau$, the total number of physical and non-physical front is finite, with a bound independent of $\tau$. So only a finite number of interactions occur in the approximate solution.
\end{itemize}
We notice that (A) is demonstrated in the next section by showing a Glimm functional $F(\tau)$ is decreasing, if the total variation of the initial perturbation is small, by specifying  $\varepsilon_0$ in Theorem \ref{T:1.1}. Then (B) follows easily.

Furthermore, to prove consistency of the approximate solutions, the key point is to show that  the total strength of the non-physical front at any non-interaction ``time" $x=\tau$ is of the order $O(\delta)$. This is achieved later by choosing carefully $\mu_\delta$, see \eqref{eq55}.
\end{remark}

To end this section, we state the following lemma, which says how to estimate the distance from a state $U_3$ by the 3-rarefaction wave curve to a state $U_1$. It can be used to prove \eqref{E:1.5} claimed in Theorem \ref{T:1.1}.

\begin{lemma}\label{l4}
Suppose that the three constant states $U_1, U_2, U_3\in
D(U_+,\delta_{0})$ satisfy $\displaystyle U_{2}=\Phi(
\alpha_1, \alpha_2, \alpha_3;U_1)$,  $U_2$ and $U_3$ are
connected by a non-physical front $\epsilon$.
Then
\begin{eqnarray}\label{E:5.2}
\left\{
\begin{array}{llll}
\displaystyle\Big|(I(q, B)+\theta)(U_3)-(I(q, B)+\theta)(U_1)\Big|
=O(1)\Big(\sum^{2}_{i=1}|\alpha_{i}|+|\alpha^{-}_{3}|+|\epsilon|\Big), \\[6pt]
\displaystyle|(q^2+\frac{2c^2}{\gamma-1})(U_3)-(q^2+\frac{2c^2}{\gamma-1})(U_1)|
=O(1)\Big(\sum^{2}_{i=1}|\alpha_{i}|+|\alpha^{-}_{3}|+|\epsilon|\Big),\\
\displaystyle\big|p_3\rho^{-\gamma}_{3}-p_{1}\rho^{-\gamma}_{1}\big|=O(1)
\Big(\sum^{2}_{i=1}|\alpha_{i}|+|\alpha^{-}_{3}|+|\epsilon|\Big),\\[6pt]
     \end{array}
     \right.
\end{eqnarray}
where $\alpha^{-}_{3}=\min\{\alpha_3, 0\}$, and $O(1)$ has a bound $C'_1$ that depends only on the background solution.
\end{lemma}

\begin{proof}
From the expression of  rarefaction wave curves, for any $\alpha_{3}\ge0$, we have
\begin{eqnarray*}
\begin{split}
&\displaystyle\Big(I(q, B)+\theta\Big)\big(\Phi_{3}(\alpha_3; U)\big)
=\Big(I(q, B)+\theta\Big)(U),\\[5pt]
&\displaystyle(q^2+\frac{2c^2}{\gamma-1})(\Phi_3(\alpha_3;
U))=(q^2+\frac{2c^2}{\gamma-1})(U),\\
&\Big(p\rho^{-\gamma}\Big)\big(\Phi_{3}(\alpha_3; U)\big)
   =\Big(p\rho^{-\gamma}\Big)(U).
\end{split}
\end{eqnarray*}
For $\alpha_3<0$ and other waves, we use the standard results, namely, employ wave curve/surface parameters $\alpha_i, i=1, 2, 3$ to measure variations of physical states (see, for example,  \cite[p.256, (5.144)]{HR2015}). 
This completes proof of the lemma.
\end{proof}

\section{Monotonicity of Glimm Functional}\label{S:5}

In this section, we construct a Glimm functional and prove its
monotonicity based on local interaction estimates. Then by induction we demonstrate that for any $\delta>0$ small, we could construct a global approximate solution $(U^\delta, g^\delta)$ as outlined in \S \ref{sec32}.

\subsection{A Glimm functional}\label{sec41}
For convenience of writing, for any weak front $\alpha$, denote its position and magnitude  by
$y_{\alpha}(x)$ and $\alpha$, respectively. Similarly, for a front $s$ of
the strong rarefaction wave, denote its location and magnitude by
$y_s(x)$ and $s$.
We also define approaching waves as follows.

\begin{definition}(Approaching waves)
\begin{itemize}
\item $(\alpha_i, \beta_j)\in \mathcal {A}_1(x)$: two weak physical fronts
$\alpha_i$ and $\beta_j$ ($i, j \in \{1, 2, 3\}$), located at
points $y_{\alpha_i}(x)$ and $y_{\beta_j}(x)$ respectively, are approaching, provided
$y_{\alpha_i}(x)< y_{\beta_j}(x)$, and satisfy that
either $i>j$, or $i=j$ and at least one of them is a shock;

\item $(\alpha_i, \epsilon)\in\mathcal{A}_2(x)$: a weak physical front $\alpha_i$ located at
$y_{\alpha_i}(x)$ is approaching a non-physical front $\epsilon$ located at $y_{\epsilon}(x)$, if $y_{\alpha_i}(x)>y_{\epsilon}(x)$.
\end{itemize}
\end{definition}

The concept here of approaching waves is the same as that introduced by Glimm in \cite{GJ}. We next introduce the most important functionals to study strong rarefaction waves. For any weak front $\alpha$ and any non-physical front $\epsilon$, denote
\[
\begin{split}
& R(x, \alpha, {\rm b})=\{s|s \text{ is  a front of the strong rarefaction wave with } y_s(x)< y_{\alpha}(x)\},\\
& R(x, \epsilon, {\rm a})=\{s|s \text{ is  a front of the strong
rarefaction wave with } y_s(x)>y_{\epsilon}(x)\},
\end{split}
\]
and define \footnote{Here and in the following, $\sum\{h(\alpha): \alpha\in \Lambda\}$ means taking sums of $h(\alpha)$ for $\alpha$ runs in the set $\Lambda$.}
\[
\begin{split}
&W(\alpha, x)=\exp\Big(K_{\omega}
\sum\{|s(x)|:s\in R(x, \alpha, {\rm b})\}\Big),\\
&W(\epsilon, x)=\exp\Big(K_{np}\sum\{|s(x)|:s\in R(x, \epsilon,
{\rm a})\}\Big).
\end{split}
\]
We observe that $R(x,\alpha,{\rm b})$ is the set of those strong
rarefaction fronts that lie below the weak front $\alpha$ at ``time"
$x$, and $R(x,\epsilon,{\rm a})$ is the set of  strong rarefaction
fronts that lie above the non-physical front $\epsilon$ at
``time" $x$. The idea is, for example, $\epsilon$ is approaching all the strong 3-rarefaction fronts in $R(x,\epsilon, \mathrm{a})$. The crucial functionals  $W(\alpha, x)$ and
$W(\epsilon, x)$, utilize fast growth of the exponential functions,  could drastically magnify the decrease of strength of the strong rarefaction wave, when a weak front penetrating into it, by further choosing the constants $K_\omega$ and $K_{np}$ large, thus overtaken
the difficulty that the 3-strong rarefaction wave has a large total variation, for which the original Glimm interaction potential (see $Q_0$ below) fails.

We further set
\begin{eqnarray*}
\begin{split}
L_{i}(x)&=\sum\big\{|\alpha_i|:\alpha_{i}\ \text{ is an $i$-weak physical wave front at ``time"} \,x \big\},\  1\leq i\leq 3,\\[5pt]
L_{4}(x)&=\sum\big\{|\epsilon|:\epsilon \ \text{ is a non-physical
front at ``time"}\, x \big \}.
\end{split}
\end{eqnarray*}
The two functionals are used to control the
total variation of all the weak waves introduced by the perturbations
of the initial data  (excluding the corner) in our problem.  We also need the following
functionals:
\begin{eqnarray*}
\begin{split}
&Q_{0}(x)=\sum\big\{|\alpha_i||\beta_j|: (\alpha_i, \beta_j)\in
\mathcal
{A}_1(x) \big\}+\sum\big\{|\alpha_i||\epsilon|:(\alpha_i, \epsilon)\in\mathcal{A}_2(x)\big\},\\[5pt]
&Q_{i}(x)=\sum\big\{|\alpha_i|W(\alpha_{i}, x): \text{ $\alpha_i$ is an $i$-weak  front}\},\quad  i=1, 2,  \\[5pt]
&Q_{4}(x)=\sum\big\{|\epsilon|W(\epsilon, x): \text{ $\epsilon$ is a non-physical front} \big\},\\[5pt]
\end{split}
\end{eqnarray*}
which represent respectively the interaction potentials between weak physical wave
fronts and/or non-physical fronts, a weak wave front and the 3-strong rarefaction waves,  a non-physical front and the 3-strong rarefaction waves. The functional $Q_0$ was originally introduced by Glimm, while $Q_{i}$ ($i=1,2,4$) resemble those appeared in \cite[p.139, (7.65)]{bressan}, and has been used in many previous works \cite{CKZ1,D1,DKZ,Zhang}. We do not need the interaction potential between 3-weak shocks and the strong 3-rarefaction waves, since for this case there are cancellations and it is not necessary to consider second-order terms in interactions of waves.

Define $\underline{S}>0$ to be the strength of the background strong rarefaction wave, namely $\Phi_3(\underline{S};U_-)=U_+$. Then we set
\begin{eqnarray*}
\begin{split}
&S(x)=\sum\{s(x): s\, \text{is a strong 3-rarefaction wave front}\},\\
&F_{1}(x)=|S(x)-\underline{S}|.
\end{split}
\end{eqnarray*}
The latter is used to measure the perturbation of  the 3-strong rarefaction wave at each ``time" $x$.

Finally we introduce
\begin{eqnarray*}
\begin{split}
L_0(x)&=\sum^{4}_{i=1}L_{i}(x),\\[3pt]
L_w(x)&=KL_{1}(x)+L_{2}(x)+K_3L_{3}(x)+L_{4}(x),\\[3pt]
Q(x)&=K_{0}Q_{0}(
x)+\sum^{2}_{i=1}K_{i}Q_{i}(x)+K_4Q_{4}(x),\\[3pt]
F_{0}(x)&=L_{w}(x)+ Q(x),
\end{split}
\end{eqnarray*}
and the Glimm functional is defined as
\begin{eqnarray*}
\begin{split}
F(x)&=F_{0}(x)+K_*F_{1}(x),
\end{split}
\end{eqnarray*}
where $K, K_0, K_1, K_2, K_3, K_4, K_\omega, K_{np}$ and $K_{*}$ are positive constants called
{\it weights} that need to be chosen later (cf. \eqref{eq424new}-\eqref{eq432}). The weight $K$ is used to handle the reflections of 1-wave fronts on free-boundary, while $K_3$ is used to magnify the cancellation between 3-weak shock fronts and 3-strong rarefaction waves. Although it turns out that we may take $K_1=K_2=K_4=K_*=1$ later, we retain them for easy to track estimates of each term in the Glimm functional in the following computations.

\subsection{Decreasing of Glimm functional}\label{sec42}

Note that the Glimm functional experiences changes only if two fronts interact, or a physical front hits the free-boundary, at some interaction ``time" $x=\tau$. We have the following crucial result.

\begin{theorem}\label{T:5.1}
There exist positive
constants $K,\ K_0,\ K_1,\ K_2,\ K_3,\ K_4,\ K_\omega,$ $K_{np},\ $ $ K_{*}$ and $\delta^*$ that depend only on the background solution and $\delta_0$, such that if
\begin{equation}\label{eq42}
F(\tau-)<\delta^*
\end{equation}
and $\delta<\delta^*$,  then for the approximate solution $(U^\delta, g^\delta)$, we have
\begin{equation}\label{E:5.16}
F(\tau+)- F(\tau-)\leq -\frac{1}{4}E_{\delta}(\tau-).
\end{equation}
\end{theorem}

Recall that $E_\delta(\tau)$ has been defined by \eqref{E:5.3}. From the definition of $F$, the assumption \eqref{eq42} and $\delta<\delta^*$ imply that

\begin{itemize}
\item[1)] There is a constant $C_0$ depending only on the background solution so that  $S(\tau-)\le C_0$;

\item[2)]  $L_0(\tau-)< \delta^*$ (note that we will take $K\ge1, K_3\ge1$ later);

\item[3)]  $U^{\delta}(\tau-, y)\in D(U_+,{\delta_{0}})$, and $|s|\le \delta<\delta^*$.
\end{itemize}

To prove \eqref{E:5.16}, we now check the six cases listed before in \S \ref{sec32}.\\

\textbf{\underline{Case 1}. Interaction between weak physical fronts.}\\

Let the two weak fronts $\alpha_i$ and $\beta_j$ interact at a point on the line $x=\tau$, and $\gamma_l$ be the generated waves, $l=1, 2, 3$, and $\epsilon$ be the outgoing non-physical front (if we use the simplified Riemann solver). By a
standard procedure (see \cite[p.133]{bressan} or \cite[p.290]{HR2015}), we have the
following estimates, even if we use the convention made in Remark \ref{rm21} on strengths of 2-waves.
\begin{lemma}\label{L:5.1}
It holds that
\begin{equation}
\epsilon=O(1)|\alpha_i||\beta_j|,\label{E:5.5}
\end{equation}
and
\begin{itemize}
\item when $i\neq j$,
\begin{equation}
\gamma_i=\alpha_i+O(1)|\alpha_i||\beta_j|, \;\; \gamma_j=\beta_j+O(1)|\alpha_i||\beta_j|;
\end{equation}
\item when $i=j$,
\begin{align}
&\gamma_l=\alpha_i+\beta_i+O(1)|\alpha_i||\beta_i|, \quad \text{for }l=i, \\
&\gamma_l=O(1)|\alpha_i||\beta_i|,\quad\qquad\qquad \text{ for } l\neq i.\label{E:5.6}
\end{align}
\end{itemize}
All the quantities $O(1)$ here are bounded in $D(U_+,\delta_0)$ with a uniform bound $C_1$.
\end{lemma}
Without loss of generality, we may assume in the following that $C_1\ge1.$

Based on the estimates \eqref{E:5.5}--\eqref{E:5.6}, we have
\[
L_k(\tau+)-L_k(\tau-)=O(1)|\alpha_i||\beta_j|,\quad k=1,2,3,4,\]
and no matter the interaction happens above/below/in the middle of the strong rarefaction waves, it always holds that
\[F_1(\tau+)-F_1(\tau-)=0.\]
Using now standard arguments as in \cite[pp.294-295]{HR2015}, and assumption \eqref{eq42}, we have
\[Q_0(\tau+)-Q_0(\tau-)\le(C_1L_0(\tau-)-1)|\alpha_i||\beta_j|\le-\frac12|\alpha_i||\beta_j|,\]
provided that
\begin{eqnarray}\label{eq47new}
\delta^*\le\frac{1}{2C_1}.
\end{eqnarray}
By bounds of $S(\tau-)$, we also get
\begin{equation*}\begin{split}
Q_{4}(\tau+)-Q_{4}(\tau-)=&\epsilon W(\epsilon,\tau+)=\epsilon W(\epsilon,\tau-)\le e^{C_0K_{np}}O(1)|\alpha_i||\beta_j|,\\
Q_{k}(\tau+)-Q_{k}(\tau-)\le&e^{C_0K_{\omega}}O(1)|\alpha_i||\beta_j|,\quad k=1,2.
\end{split}\end{equation*}
It follows that
\begin{equation*}\begin{split}
Q(\tau+)-Q(\tau-)\leq&\Big(C_1((K_1+K_2)e^{C_0K_{\omega}}+K_4e^{C_0K_{np}})-\frac12 K_0\Big)|\alpha_i||\beta_j|,\\
F(\tau+)-F(\tau-)\leq& \Big(C_1\big(K+K_3+3+(K_1+K_2)e^{C_0K_{\omega}}+K_4e^{C_0K_{np}}\big)\nonumber\\
&\qquad-\frac12 K_0\Big)|\alpha_i||\beta_j|
\le -\frac{1}{4}|\alpha_i||\beta_j|,
\end{split}\end{equation*}
where we choose $K_0$ large enough so that
\begin{eqnarray}
C_1\big(K+K_3+3+(K_1+K_2)e^{C_0K_{\omega}}+K_4e^{C_0K_{np}}\big)-\frac12 K_0\le -\frac12.\label{eq411}
\end{eqnarray}
\vspace{0.3cm}

\textbf{\underline{Case 2}. Reflection of a front on the free-boundary.}\\

Assume that a weak $1$-wave front $\alpha_1$ hits the characteristic free-boundary at a point $(\tau, g^\delta(\tau))$. For the accurate Riemann solver, denote the reflected wave
by $\beta_3$. From Lemma \ref{l:4.3}, we have
\[
\begin{split}
L_1(\tau+)-L_1(\tau-)=&-|\alpha_1|, \quad L_i(\tau+)-L_i(\tau-)=0,\; i=2, 4,\\[5pt]
L_3(\tau+)-L_3(\tau-)=&K_{b_1}|\alpha_1|, \quad F_1(\tau+)-F_1(\tau-)=0,\\[5pt]
\end{split}
\]
and
\begin{eqnarray*}
&&Q_0(\tau+)-Q_0(\tau-)=|\beta_3|L_0(\tau_-)=K_{b1}L_0(\tau_-)|\alpha_1|,\\
&&Q_k(\tau+)-Q_k(\tau-)=0, \quad k=1,2, 4.
\end{eqnarray*}
It follows that
\begin{eqnarray*}
F(\tau+)-F(\tau-)=\big(-K+(K_0L_{0}(\tau-)+K_3)K_{b1}\big)|\alpha_1|\le-\frac14|\alpha_1|,
\end{eqnarray*}
if (recall that $K_{b_1}\le C_b$)
\begin{equation}\label{eq413*}
K\geq C_{b}(K_3+K_0L_0(\tau-))+\frac14.
\end{equation}

If the simplified Riemann solver is used, then $\beta_3$ is replaced by a non-physical front $\epsilon$, and from \eqref{eq36}, we have
\[
\begin{split}
L_1(\tau+)-L_1(\tau-)=&-|\alpha_1|, \quad L_i(\tau+)-L_i(\tau-)=0,\; i=2, 3,\\[5pt]
L_4(\tau+)-L_4(\tau-)=&C'_1K_{b_1}|\alpha_1|, \quad F_1(\tau+)-F_1(\tau-)=0,\\[5pt]
\end{split}
\]
as well as
\begin{eqnarray*}
&&Q_0(\tau+)-Q_0(\tau-)\le|\epsilon|L_0(\tau-)=C'_1K_{b1}L_0(\tau-)|\alpha_1|,\\
&&Q_k(\tau+)-Q_k(\tau-)=0, \quad k=1,2,\\
&&Q_4(\tau+)-Q_4(\tau-)=|\epsilon|W(\epsilon,\tau-)\le e^{K_{np}C_0}C'_1K_{b_1}|\alpha_1|.
\end{eqnarray*}
Therefore
\begin{eqnarray*}
F(\tau+)-F(\tau-)=\big(-K+(K_0L_{0}(\tau-)+1+K_4e^{K_{np}C_0})C'_1K_{b1}\big)|
\alpha_1|\le-\frac14|\alpha_1|,
\end{eqnarray*}
if
\begin{equation}\label{eq413}
K\geq C'_1C_{b}(1+K_3+K_4e^{K_{np}C_0}+K_0L_0(\tau-))+\frac14.
\end{equation}
Without loss of generality, we may assume $C'_1>1$ in \eqref{eq36}, so this condition implies \eqref{eq413*}.

\vspace{0.2cm}

\textbf{\underline{Case 3}. Interaction between a strong 3-rarefaction front and 1- or 2-weak fronts from above.}\\[-0.3cm]

In this case, without loss of generality, we consider a 3-strong
rarefaction front $s$ interacts with a 1-weak  front $\alpha_1$. It
is similar to analyze the other situation. Let the below and above
states of the 3-strong rarefaction  front $s$ be $U_l$ and $U_m$,
respectively. The incoming 1-wave front $\alpha_1$ connects the
states $U_m$ and $U_r$. Denote the outgoing waves by
$\gamma_j$ ($1\leq j\leq 2$), $s'$, and the non-physical front by $\epsilon$, respectively.

The following lemma is standard \cite[p.133, Lemma 7.2]{bressan}.
\begin{lemma}\label{lem42}
We have the estimates:
\[
\begin{split}
&\gamma_1=\alpha_1+O(1)|\alpha_1||s|,\quad \gamma_2=O(1)|\alpha_1||s|, \\
&s'=s+O(1)|\alpha_1||s|, \quad \epsilon=O(1)|\alpha_1||s|,
\end{split}
\]
where $O(1)$ depend only on the background solution, with a bound $C_1$.
\end{lemma}

Based on this lemma, we get
\begin{eqnarray*}
&&L_k(\tau+)-L_k(\tau-)=O(1)|\alpha_1||s|, \quad k=1,2,4;\\
&& L_3(\tau+)-L_3(\tau-)=0,
\end{eqnarray*}
and
\begin{eqnarray*}
S(\tau+)-S(\tau-)=s'-s=O(1)|\alpha_1||s|,
\end{eqnarray*}
which implies that
\begin{eqnarray*}
F_1(\tau+)-F_1(\tau-)=O(1)|\alpha_1||s|
\end{eqnarray*}
by triangle inequality.
It is easy to see that
\begin{eqnarray*}
Q_0(\tau+)-Q_0(\tau-)=4O(1)|\alpha_1||s|L_0(\tau-).
\end{eqnarray*}

The estimates of $Q_k$ ($k=1,2,4$) are more complicated. Let $\epsilon'$ be any non-physical front lying below the interaction point at ``time" $\tau$, and $S_{\epsilon'}(\tau)$ the total strength of strong 3-rarefaction fronts lying above $\epsilon'$.
Then $$S_{\epsilon'}(\tau+)-S_{\epsilon'}(\tau-)=s'-s,$$ and
\begin{equation*}\begin{split}
&|W(\epsilon',\tau+)-W(\epsilon',\tau-)|\nonumber\\
=&|e^{K_{np}S_{\epsilon'}
(\tau+)}-e^{K_{np}S_{\epsilon'}(\tau-)}|=|e^{K_{np}S_{\epsilon'}
(\tau-)}
\left(e^{K_{np}(s'-s)}-1\right)|\\
=&|e^{K_{np}S_{\epsilon'}
(\tau-)}\left(e^{K_{np} O(1)|\alpha_1||s|}-1\right)|\le 3e^{K_{np}C_0}O(1)K_{np}|\alpha_1||s|.
\end{split}\end{equation*}
Here we assumed that
\begin{eqnarray}\label{eq414}
C_1K_{np}L_1(\tau-)\delta<1,
\end{eqnarray}
which implies that
\begin{eqnarray*}
|O(1)|K_{np}|\alpha_1||s|<1,
\end{eqnarray*}
and used the fact that $|e^x-1|\le 3|x|$ for $|x|\le1$.
It follows that
 \begin{equation*}\begin{split}
Q_4(\tau+)-Q_4(\tau-)=&\epsilon W(\epsilon,\tau+)+\sum_{\epsilon'} \epsilon'(W(\epsilon',\tau+)-W(\epsilon',\tau-))\\
\le&\epsilon W(\epsilon,\tau+)+3L_4(\tau-)e^{K_{np}C_0}O(1)K_{np}|\alpha_1||s|\\
=&O(1)|\alpha_1||s|e^{K_{np}C_0}+3L_4(\tau-)e^{K_{np}C_0}O(1)K_{np}|\alpha_1||s|\\
=&O(1)|\alpha_1||s|e^{K_{np}C_0}(1+3L_4(\tau-)K_{np}).
\end{split}\end{equation*}

Similarly, by considering all those weak fronts $\beta_2$ lying above $\alpha_1$ at $x=\tau$, we could obtain
  \begin{equation*}\begin{split}
Q_2(\tau+)-Q_2(\tau-)\le&|\gamma_2| W(\gamma_2,\tau+)+3L_2(\tau-)e^{K_{\omega}C_0}K_{\omega}O(1)|\alpha_1|
|s|\\
=&|\gamma_2| W(\alpha_1,\tau-)e^{-K_\omega s}+3L_2(\tau-)e^{K_{\omega}C_0}K_{\omega}O(1)|\alpha_1||s|\\
\le&C_1|\alpha_1||s|(W(\alpha_1,\tau-)+3L_2(\tau-)K_{\omega}e^{K_{\omega}C_0}),
 \end{split}\end{equation*}
provided that
\begin{eqnarray}\label{eq415}
C_1K_{\omega}L_2(\tau-)\delta<1.
\end{eqnarray}
We notice that the term $W(\alpha_1,\tau-)$ shall be retained to balance a term appeared in $Q_1$ below.

Now we turn to $Q_1$. We firstly observe that by definition,
\begin{eqnarray}
W(\gamma_1,\tau+)e^{K_{\omega}s}=W(\alpha_1, \tau-).
\end{eqnarray}
Then
\begin{equation*}\begin{split}
&|\gamma_1| W(\gamma_1,\tau+)-|\alpha_1| W(\alpha_1,\tau-)\\
=&(|\alpha_1|+O (1)|\alpha_1||s|)W(\alpha_1,\tau-)e^{-K_{\omega}s}-|\alpha_1|W(\alpha_1,\tau-)\\
=&|\alpha_1|W(\alpha_1,\tau-)\big(e^{-K_{\omega} s}-1+O(1)|s|e^{-K_{\omega} s}\big)\le|\alpha_1||s|W(\alpha_1,\tau-)(-\frac12 K_\omega+C_1).
 \end{split}\end{equation*}
Here in the second last inequality we used $e^{-x}-1<-\frac12 x$ for $1>x>0$, and $O(1)|s|e^{-K_{\omega} s}\le C_1|s|$ since $s>0$.  The weights will be chosen independent of $\delta$, so we may choose $\delta$ small such that $K_\omega s<1.$  We remark that the computation carried here is one of the key point in dealing with large rarefaction waves.

Then considering any  1-wave front $\beta_1$ lying above $\alpha_1$ at $x=\tau$, like what we obtained before, one has
\begin{eqnarray*}
\sum_{\beta_1}|\beta_1|(W(\beta_1,\tau+)-W(\beta_1,\tau-))
\le 3L_1(\tau-)K_\omega O(1)|\alpha_1||s|e^{K_\omega C_0},
\end{eqnarray*}
provided that \eqref{eq415} holds. Hence we have
\begin{eqnarray*}
Q_1(\tau+)-Q_1(\tau-)\le|\alpha_1||s|\Big(
W(\alpha_1,\tau-)(-\frac12K_\omega+C_1)+3C_1L_1(\tau-)K_\omega e^{K_\omega C_0}\Big).
\end{eqnarray*}
Here we also used the assumption \eqref{eq426}.

Finally, from the above estimates of $Q_k$, we get
 \begin{equation*}\begin{split}
&Q(\tau+)-Q(\tau-)\nonumber\\
\le&4C_1K_0|\alpha_1||s|L_0(\tau-)+ K_1 |\alpha_1||s|\Big(W(\alpha_1,\tau-)\big(-\frac{1}{2}K_\omega+C_1\big)
\nonumber\\
&+3C_1L_0(\tau-)K_\omega e^{K_\omega C_0}\Big)\nonumber\\
&+C_1K_2|\alpha_1||s|\Big( W(\alpha_1,\tau-)+3L_0(\tau-)K_\omega e^{K_\omega C_0}\Big)\nonumber\\
&+C_1 K_4|\alpha_1||s|e^{K_{np} C_0}(1+3L_0(\tau-)K_{np})\nonumber\\
=&|\alpha_1||s|C_1\Big(L_0(\tau-)\big(4K_0+3K_4K_{np}e^{K_{np}C_0}
+3(K_1+K_2)K_{\omega}e^{K_\omega C_0}\big)\nonumber\\
&+\big(K_1(1-\frac{1}{2}\frac{K_\omega}{C_1})+K_2\big)W(\alpha_1,\tau-)+K_4e^{K_{np} C_0}\Big)\\
\le&|\alpha_1||s|C_1\Big(L_0(\tau-)\big(4K_0+3K_4K_{np}e^{K_{np}C_0}
+3(K_1+K_2)K_{\omega}e^{K_\omega C_0}\big)\nonumber\\
&+\underline{\big(K_1(1-\frac{1}{2}\frac{K_\omega}{C_1})+K_2\big)}+K_4e^{K_{np} C_0}\Big),
\end{split}\end{equation*}
here for the last inequality, we used
$W(\alpha_1,\tau-)\ge1$, and the term with underline will be negative.
Therefore
\begin{equation*}\begin{split}
&F(\tau+)-F(\tau-)\nonumber\\
\le & C_1|\alpha_1||s|(K+3+K_*)+Q(\tau+)-Q(\tau-)\nonumber\\
=&C_1|\alpha_1||s|\Big(L_0(\tau-)\big(4K_0+3K_4K_{np}e^{K_{np}C_0}
+3(K_1+K_2)K_{\omega}e^{K_\omega C_0}\big)\nonumber\\
&+\big(K_1(1-\frac{1}{2}\frac{K_\omega}{C_1})+K_2\big)+K_4e^{K_{np} C_0}+K+3+K_*\Big).
 \end{split}\end{equation*}
We wish to choose $K_\omega$ large so that
\begin{eqnarray}\label{eq417}
&&L_0(\tau-)\big(4K_0+3K_4K_{np}e^{K_{np}C_0}
+3(K_1+K_2)K_{\omega}e^{K_\omega C_0}\big)\nonumber\\
&&+\big(K_1(1-\frac{1}{2}\frac{K_\omega}{C_1})+K_2\big)+K_4e^{K_{np} C_0}+K+3+K_*\le-\frac12,
\end{eqnarray}
which implies that
\begin{eqnarray*}
\begin{split} F(\tau+)-F(\tau-)\leq-\frac{1}{4}|\alpha_1||s|.
\end{split}
\end{eqnarray*}
Here we used the fact that $C_1\ge1$.

If we consider a weak 2-wave front interacts with a 3-strong rarefaction front, we may get similar estimates as above, provided that
\begin{eqnarray}
&&L_0(\tau-)\big(4K_0+3K_4K_{np}e^{K_{np}C_0}
+3(K_1+K_2)K_{\omega}e^{K_\omega C_0}\big)\nonumber\\
&&+\big(K_2(1-\frac{1}{2}\frac{K_\omega}{C_1})+K_1\big)+K_4e^{K_{np} C_0}+K+3+K_*\le-\frac12.\label{eq419}
\end{eqnarray}

\textbf{\underline{Case 4}. Interaction between a 3-strong rarefaction front and a 3-weak shock front from the below (above).}\\[-0.3cm]

Suppose that a 3-strong rarefaction front $s>0$ and a 3-weak shock
front $\alpha_3<0$ interact at a point on $x=\tau$. Let $\gamma_k$
and $\epsilon$ be the outgoing $k$-waves $(k=1,2,3)$ and
non-physical front, respectively.
\begin{lemma}\label{lem43}
We have the following estimates (cf. \cite[p.133, Lemma 7.2]{bressan}):
\[
\begin{split}
&\gamma_i=O(1)|\alpha_3||s|,\quad i=1, 2, \\
&\gamma_3=\alpha_3+s+O(1)|\alpha_3||s|, \quad
\epsilon=O(1)|\alpha_3||s|.
\end{split}
\]
\end{lemma}

Depending on whether the produced 3-wave is a rarefaction wave or a shock,  we consider the two subcases.

\textbf{\underline{Case 4.1}.} $\gamma_3\geq 0$.

In this case,
\[
|\gamma_3|=|s|-|\alpha_3|+O(1)|\alpha_3||s|,
\]
and
\[
\begin{split}
&L_i(\tau+)-L_i(\tau-)=O(1)|\alpha_3||s|, \quad i=1, 2, 4,\\[5pt]
&L_3(\tau+)-L_3(\tau-)=-|\alpha_3|.\\[5pt]
\end{split}
\]
Since
\[S(\tau+)-S(\tau-)=|\gamma_3|-|s|=-|\alpha_3|+O(1)|\alpha_3||s|,\]
by triangle inequality,
\[F_1(\tau+)-F_1(\tau-)\le |\alpha_3|(1-O(1)|s|)\le\frac32|\alpha_3|.\]
We do not know exactly the interaction potential of $\alpha_3$ at $x=\tau-$ , but which is anyway nonnegative. Hence we have
\[Q_0(\tau+)-Q_0(\tau-)\le 3O(1)L_0(\tau-)|\alpha_3||s|.\]

Now for $k=1,2,$ let $\beta_k$ be any weak front lying above $\alpha_3$ at $x=\tau-$, and $S'$ the total strength of strong 3-rarefaction fronts lying below $\beta_k$ at $x=\tau-$. Then
\begin{equation*}\begin{split}
W(\beta_k,\tau+)-W(\beta_k,\tau-)=&e^{K_\omega(S'-s+\gamma_3)}
-e^{K_\omega S'}\le e^{K_\omega S'}\Big(e^{-K_{\omega}|\alpha_3|(1-O(1)|s|)}-1\Big)\\
\le& e^{-K_{\omega}|\alpha_3|(1-O(1)|s|)}-1\le-\frac14 K_\omega|\alpha_3|.
\end{split}\end{equation*}
Here we used the fact that $e^{-x}-1\le-\frac12 x$ for $0<x<1$, and the assumption that
\begin{eqnarray}\label{eq420}
C_1\delta<\frac12, \quad K_\omega L_3(\tau-)<1.
\end{eqnarray}
It follows that
\begin{equation*}\begin{split}
Q_k(\tau+)-Q_k(\tau-)=&|\gamma_k|W(\gamma_k,\tau+)+\sum_{\beta_k}
|\beta_k|\Big(W(\beta_k,\tau+)-W(\beta_k,\tau-)\Big)\\
\le&O(1)|\alpha_3||s|e^{K_\omega C_0}.
\end{split}\end{equation*}
Similarly, we have
\begin{eqnarray*}
Q_4(\tau+)-Q_4(\tau-)\le|\epsilon|W(\epsilon,\tau+)\le C_1|\alpha_3||s|e^{K_{np} C_0},
\end{eqnarray*}
provided that $K_{np} L_4(\tau-)<1$.

Summing up, one has
\begin{eqnarray*}
Q(\tau+)-Q(\tau-)=|\alpha_3|\Big(3C_1K_0L_0(\tau-)|s|+\sum_{j=1}^2K_j
C_1|s|e^{K_\omega C_0}+K_4C_1|s|e^{K_{np}C_0}\Big),
\end{eqnarray*}
and
\begin{equation*}\begin{split}
F(\tau+)-F(\tau-)\le&|\alpha_3|\Big(3C_1K_0L_0(\tau-)|s|+\sum_{j=1}^2K_j
C_1|s|e^{K_\omega C_0}\\
&+K_4C_1|s|e^{K_{np}C_0}+\frac32K_*+(K+3)C_1|s|-K_3\Big)\\
\le&-\frac14|\alpha_3|,
 \end{split}\end{equation*}
if $K_3$ is sufficiently large so that
 \begin{align}\label{eq421}
K_3>&\frac14+\frac32K_*+\Big(3C_1K_0L_0(\tau-)+\sum_{j=1}^2K_j
C_1e^{K_\omega C_0}\nonumber\\
&+K_4C_1e^{K_{np}C_0}+(K+3)C_1\Big)|s|.
 \end{align}

\textbf{\underline{Case 4.2}.} $\gamma_3< 0$.

For this case,
\[
|\gamma_3|=|\alpha_3|-|s|+O(1)|\alpha_3||s|,
\]
and
\[
\begin{split}
&L_i(\tau+)-L_i(\tau-)=O(1)|\alpha_3||s|, \quad i=1, 2, 4,\\[5pt]
&L_3(\tau+)-L_3(\tau-)=-|s|+O(1)|\alpha_3||s|,\\[5pt]
&S(\tau+)-S(\tau-)=-|s|,\quad F_1(\tau+)-F_1(\tau-)\le|s|.\\[5pt]
\end{split}
\]
Direct calculation yields
 \begin{equation*}\begin{split}
Q_0(\tau+)-Q_0(\tau-)\le& 5O(1)|\alpha_3||s|L_0(\tau-),\\
Q_{4}(\tau+)-Q_{4}(\tau-)=& O(1)|\alpha_3||s|e^{K_{np}C_0}+\sum_{\epsilon'}|\epsilon'|
\Big(W(\epsilon',\tau+)-W(\epsilon',\tau-)\Big)\\
\le&O(1)|\alpha_3||s|e^{K_{np}C_0}.
\end{split}\end{equation*}
Here the summation is over all those non-physical fronts $\epsilon'$ lying below $\epsilon$ at $x=\tau$. Suppose $S'$ is the total strength of strong 3-rarefaction fronts lying above $\epsilon'$. We used the fact that
\begin{eqnarray*}
W(\epsilon',\tau+)-W(\epsilon',\tau-)=e^{K_{np}S'}(1-e^{K_{np}s})\le 1-e^{K_{np}s}\le -K_{np}s<0.
\end{eqnarray*}
Similarly, for $k=1,2,$
\begin{eqnarray*}
Q_{k}(\tau+)-Q_{k}(\tau-)
\le O(1)|\alpha_3||s|e^{K_{\omega}C_0}.
\end{eqnarray*}
Then we conclude
 \begin{equation*}\begin{split}
 Q(\tau+)-Q(\tau-)\le&|s|\Big(C_1|\alpha_3|\big(5K_0L_0(\tau-)+K_4e^{K_{np}C_0}+
 (K_1+K_2)e^{K_\omega C_0}\big)\Big),
\end{split}\end{equation*}
and
 \begin{equation*}\begin{split}
F(\tau+)-F(\tau-)\le&|s|\Big(C_1|\alpha_3|\big(5K_0L_0(\tau-)+K_4e^{K_{np}C_0}+
 (K_1+K_2)e^{K_\omega C_0}\big)\nonumber\\
 &+K_*+(K+3)C_1|\alpha_3|-K_3(1-C_1|\alpha_3|)\Big).
\end{split}\end{equation*}
So if
\begin{eqnarray}
&&C_1L_3(\tau-)\le \frac12,\label{eq422}\\
&&K_3>\frac12+2\Big(C_1|\alpha_3|\big(K+3+5K_0L_0(\tau-)+K_4e^{K_{np}C_0}+
 (K_1+K_2)e^{K_\omega C_0}\big)\nonumber\\&&\qquad\qquad
 +K_*\Big),\label{eq423}
\end{eqnarray}
there follows
\begin{equation}
F(\tau+)-F(\tau-)<-\frac{1}{4}|s|.
\end{equation}

\textbf{\underline{Case 5}. Interaction between a 3-strong rarefaction front and a non-physical front from the below.}\\[-0.3cm]

Suppose that a front $s$ of the 3-strong rarefaction waves and a
non-physical front $\epsilon$ interact when $x=\tau$. Let $s_0$ and
$\epsilon_0$ be respectively the outgoing rarefaction wave front and non-physical
front. Then we have the following standard lemma.
\begin{lemma}\label{lem44}
There hold (see \cite[p.133, Lemma 7.2]{bressan})
\[
\begin{split}
&s_0=s,\quad \epsilon_0=\epsilon+O(1)|\epsilon||s|.
\end{split}
\]
\end{lemma}

It is clear that
\begin{eqnarray*}
&&L_k(\tau+)-L_k(\tau-)=0, \quad k=1,2,3,\\
&&L_4(\tau+)-L_4(\tau-)=O(1)|\epsilon||s|,\\
&&S(\tau+)-S(\tau-)=0,\quad F_1(\tau+)-F_1(\tau-)=0,\\
&&Q_0(\tau+)-Q_0(\tau-)=O(1)|\epsilon||s|L_0(\tau-).
\end{eqnarray*}
To calculate $Q_k$ for $k=1,2$, note that since $s_0=s$, we always have $W(\beta_k,\tau+)=W(\beta_k,\tau-)$, for any weak $k$-front $\beta_k$ lying above $s$ at $x=\tau$. Since no weak front from $k$-th family is involved in this case, we have
\begin{eqnarray*}
Q_k(\tau+)-Q_k(\tau-)=0, \quad k=1,2.
\end{eqnarray*}
Similarly, for any non-physical front $\epsilon'$ lying below $\epsilon$, we have $W(\epsilon',\tau+)=W(\epsilon',\tau-)$. This implies that
\begin{equation*}\begin{split}
Q_4(\tau+)-Q_4(\tau-)=&|\epsilon_0|W(\epsilon_0,\tau+)-|\epsilon| W(\epsilon, \tau-)\\
=&(|\epsilon|+O(1)|\epsilon||s|)W(\epsilon_0,\tau+)-|\epsilon| W(\epsilon_0,\tau+)e^{K_{np}s}\nonumber\\
\le&|\epsilon| W(\epsilon_0,\tau+)(C_1 s+1-e^{K_{np}s})\le |\epsilon| (C_1 s+1-e^{K_{np}s})\\
\le& |\epsilon| |s|(C_1-K_{np})<0.
\end{split}\end{equation*}
Here we used $W(\epsilon_0,\tau+)\ge1$,  $1-e^x<-x$ for $x>0$, and positiveness of $\epsilon, s$.
It follows that
\[
\begin{split}
&Q(\tau+)-Q(\tau-)\le |\epsilon|| s|(C_1K_0L_0(\tau-)+K_4(C_1-K_{np})),\\
&F(\tau+)-F(\tau-)\le |\epsilon|| s|(C_1+C_1K_0L_0(\tau-)+K_4(C_1-K_{np})).\\
\end{split}
\]
Therefore, if $K_{np}$ is large enough so that
\begin{eqnarray}\label{eq425}
K_4(K_{np}-C_1)>1+C_1+C_1K_0L_0(\tau-),
\end{eqnarray}
then
\[ F(\tau+)-F(\tau-)\leq -\frac{1}{4}|\epsilon||s|.
\]

\textbf{\underline{Case 6}. Interaction between a non-physical front and an $i$-weak front
from the above $(i=1, 2, 3)$.}\\[-0.1cm]

Suppose that a non-physical front $\epsilon$ interacts with an
$i$-weak front $\alpha_i$ $(i=1, 2, 3)$ from the above at some point on $\{x=\tau\}$.
Let the outgoing waves be $\gamma_i$ and $\epsilon'$
respectively. From the definition of the simplified Riemann solver,
we have the following lemma (cf. \cite[p.133, Lemma 7.2]{bressan}).

\begin{lemma}\label{L:4.8}
It holds that
\[
\gamma_i=\alpha_i, \qquad
\epsilon'=\epsilon+O(1)|\alpha_i||\epsilon|,
\]
where $O(1)$ is bounded, with a bound $C_1$ depending only on the background solution.
\end{lemma}

On this occasion, we easily deduce that
\begin{eqnarray*}
&&L_k(\tau+)-L_k(\tau-)=0,\quad k=1,2,3;\\
&&L_4(\tau+)-L_4(\tau-)= O(1)|\alpha_i||\epsilon|,\\
&&S(\tau+)-S(\tau-)=0,\quad F_1(\tau+)-F_1(\tau-)=0,\\
&&Q_0(\tau+)-Q_0(\tau-)= O(1)|\alpha_i||\epsilon| L_0(\tau-)-|\alpha_i||\epsilon|,\\
&&Q_k(\tau+)-Q_k(\tau-)= 0,\quad k=1,2,
\end{eqnarray*}
and
\begin{equation*}\begin{split}
Q_4(\tau+)-Q_4(\tau-)
= &W(\epsilon',
\tau+)|\epsilon'|-W(\epsilon, \tau-)|\epsilon|\\
= &(|\epsilon'|-|\epsilon|)W(\epsilon,
\tau-)\\
= &O(1)|\alpha_i|\epsilon|W(\epsilon, \tau-)\le C_1|\alpha_i||\epsilon|e^{k_{np}C_0}.
\end{split}\end{equation*}
Therefore,  we have
\[
Q(\tau+)-Q(\tau-)\le|\alpha_i||\epsilon|\big((C_1L_0(\tau-)-1)K_0+C_1e^{K_{np}C_0}K_4\big),
\]
and
\[
F(\tau+)-F(\tau-)\le|\alpha_i||\epsilon|\Big((C_1L_0(\tau-)-1)K_0+C_1e^{K_{np}C_0}K_4+C_1\Big).
\]
So we shall choose $K_0, K_4, L_0(\tau-)$ satisfying
\begin{eqnarray}
&&C_1L_0(\tau-)<\frac{1}{20},\label{eq426}\\
&&K_0>1+2C_1(1+K_4e^{K_{np}C_0})\label{eq427}
\end{eqnarray}
to get
\[
F(\tau+)-F(\tau-)\leq -\frac{1}{4}|\alpha_i||\epsilon|.
\]

\vspace{0.4cm}
Finally we choose $\delta^*$ and various weights. We note that \eqref{eq427} is guaranteed by \eqref{eq411}. By \eqref{eq425} and \eqref{eq426}, we may take
\begin{eqnarray}\label{eq424new}
&&K_4=1,\quad K_{np}=2+3C_1
\end{eqnarray}
if \eqref{eq429} holds.
Set
\begin{eqnarray}\label{eq428}
K_*=1, \qquad K_1=K_2=1,
\end{eqnarray}
and from \eqref{eq421}\eqref{eq423}, we choose
\begin{eqnarray}\label{eq426new}
K_3=5,
\end{eqnarray}
provided that
\begin{eqnarray}
&&|s|\le\frac{1}{C_1}\Big(K+6+e^{K_{np}C_0}+2e^{K_\omega C_0}\Big)^{-1},\\
&&|\alpha_3|\le (C_1)^{-1}\Big((K+8)+e^{K_{np}C_0}+2e^{K_\omega C_0}\Big)^{-1},\\
&&K_0L_0(\tau-)\le1.\label{eq429}
\end{eqnarray}
For \eqref{eq413} to be true, we take
\begin{eqnarray}\label{430}
K=C'_1C_b(7+e^{K_{np}C_0})+1.
\end{eqnarray}
Then \eqref{eq411} holds if
\begin{eqnarray}\label{eq431}
K_0=2+2C_1(K+10+e^{C_0K_{np}}+2e^{K_\omega C_0}).
\end{eqnarray}
In the last, from \eqref{eq417}\eqref{eq419} and \eqref{eq428}, we may set
\begin{eqnarray}\label{eq432}
K_\omega=2C_1\big(K+8+e^{C_0K_{np}}\big),
\end{eqnarray}
and require that
\begin{eqnarray}
L_0(\tau-)\le \big(4K_0+3K_{np}e^{C_0K_{np}}+6K_\omega e^{K_\omega C_0}\big)^{-1}
\end{eqnarray}
to guarantee the inequalities \eqref{eq417}\eqref{eq419}.
Therefore we shall choose (see also \eqref{eq47new}, \eqref{eq414}, \eqref{eq425}, \eqref{eq420})
\begin{align}\label{eq434}
\delta^*=&\min\left\{\frac{1}{20C_1(1+K_0+K_\omega+K_{np})}, (4K_0+3K_{np}e^{C_0K_{np}}+6K_\omega e^{K_\omega C_0})^{-1},\right.\nonumber\\
&\left.\frac{1}{\sqrt{C_1(K_{np}+K_\omega)}},
 C_1^{-1}(K+8+e^{K_{np}C_0}+2 e^{K_\omega C_0})^{-1}\right\},
\end{align}
and $F(\tau-)\le\delta^*, \delta<\delta^*$ make all the above estimates valid. Recall that $\delta$ is set so that each rarefaction front has strength at most $\delta$ in the approximate solution $(U^\delta, g^\delta)$.

This completes the proof of Theorem \ref{T:5.1}.
\vspace{0.2cm}

\subsection{Finiteness of fronts and interactions }\label{sec43}

For the $\delta^*$ determined in Theorem \ref{T:5.1}, and any $\delta<\delta^*$, suppose $(U^\delta,g^\delta)$ is an approximate solution constructed by the front tracking algorithm. Let $0<\tau_1<\tau_2<\cdots<\tau_k<\cdots$ be the interaction ``time". Note that for $x=0$, we solve all the Riemann problems by accurate Riemann solver. Particularly, by Lemma \ref{lem23}, we have
\[
|S(0+)-\underline{S}|\leq C|U^{\delta}(0,0+)-U_{+}|\leq C{\rm{TV. }}U^{\delta}(0,y)\le   C_3\varepsilon\le C_3\varepsilon_0.
\]
Also recall that $\sum_{i=1}^4 L_i(0+)\le C_3\mathrm{TV.}U^\delta(0,y)\le   C_3\varepsilon\le C_3\varepsilon_0$ by property of Riemann problems and the assumption in Theorem \ref{T:1.1}; and by definition, $Q_0(0+)\le (\sum_{i=1}^4 L_i(0+))^2$. It follows that
\begin{eqnarray*}
F({\tau_1}-)=F(0+)\le C_3(\varepsilon+\varepsilon^2)\le C_3(\varepsilon_0+\varepsilon_0^2).
\end{eqnarray*}
All the constants $C_3$ appeared here depend only on the Euler system and the background solution, through the constants $C_1$ and weights chosen in previous sections. Hence we could choose $\varepsilon_0\le1$ claimed in Theorem \ref{T:1.1} so that $C_3(\varepsilon_0+\varepsilon_0^2)<\delta^*$. By Theorem \ref{T:5.1}, we infer that
\begin{eqnarray}\label{eq435}
F(\tau)<F({\tau_1}-)=F(0+)\le \min\{\delta^*, 2C_3\varepsilon\}
\end{eqnarray}
for any approximate solution $(U^\delta, g^\delta)$ defined on $0\le x\le\tau$.

\begin{corollary}\label{cor41}
For given $\delta\in(0,\delta^*)$, the number of fronts in the approximate solution $(U^\delta,g^\delta)$ at ``time" $x=\tau$ is bounded by a constant independent of $\tau$, and there are only a finite number of interactions of fronts.
\end{corollary}

\begin{proof}
Let $N(\tau)$ be the number of fronts at $x=\tau$. Then $N_0=N(0+)$ is finite and depends on the number of initial jumps in $U^\delta(0,y)$ and $\varepsilon_0, \delta$.

The changes of $N(\tau)$ can only occur in two situations.

1) The accurate Riemann solver is used at an interaction point $(\tau,y)$. In this case, we have $F(\tau+)-F(\tau-)\le-\frac14 E_\delta(\tau-)\le-\frac14\mu_\delta$. The constant $\mu_\delta$, independent of $\tau$, is chosen by \eqref{eq55}. Since $F$ is nonnegative, this situation occurs at most $4\delta^*/\mu_\delta$ times, and each time, the number of new-born wave fronts (they are physical fronts) is at most $O(3\delta^*/\delta)$.

2) The simplified Riemann solver is used. For Cases 2,4,5,6, the number of fronts does not change. For Cases 1, 3, only one new non-physical front is born. Therefore, we know  the number of physical wave fronts is still $N_p=N_p(\delta^*, \delta, \mu_\delta, N_0)$. Now, two physical fronts can only interact at most twice (possibly one before and one after reflection from the free-boundary). Hence the total number of new non-physical fronts is at most $N_p+2N_p^2$. Therefore we have
\begin{eqnarray*}
N(\tau)\le 2N_p+2N_p^2,\quad \forall \tau>0.
\end{eqnarray*}
Since the number of fronts is finite, and each pair of fronts can meet at most twice, there are only finite interactions.
\end{proof}

\subsection{Uniform estimates of total variations}

To prove \eqref{eq:2.1.12}, we need to show that
\begin{eqnarray}\label{eq437}
|\mathrm{TV. }\{p^\delta(x,\cdot):[g^\delta(x),+\infty)\}-\mathrm{TV. }\{p_{\mathrm{b}}(x,\cdot):[g_\mathrm{b}(x),+\infty)\}|\leq
M_1\varepsilon.
\end{eqnarray}

Recall that an approximate solution $U^\delta(x,y)$ is piecewise constant, with discontinuities across finite (say $N$) fronts $y=y_k(x)$, $k=1,\ldots,N$, which are labeled so that $y_1(x)>y_2(x)>\cdots>y_N(x)=g^\delta(x)$. Each $y_k$ connects the state $U_{k-1}$ above it to the state $U_k$ below it.
Let $H(s)$ be the Heavside step function, whose value is zero for negative argument and one for positive argument. Then we could write
\begin{eqnarray}
U^\delta(x,y)=U_0+\sum_{k=1}^{N-1}(U_k-U_{k-1})H(y_k(x)-y),\quad \text{for} \ y\ge g^\delta(x).
\end{eqnarray}
Now set
\begin{eqnarray*}
\Theta_k=\begin{cases}
1,&y_k\ \ \text{is a weak wave front or non-physical front,}\\
0, &y_k\ \ \text{is a 3-strong rarefaction  front.}
\end{cases}
\end{eqnarray*}
We may write $U^\delta(x,\cdot)=U^\delta_w(x,\cdot)+U^\delta_s(x,\cdot)$, with $$U^\delta_w(x,\cdot)=U_0+\sum_{k=1}^{N-1}(U_k-U_{k-1})\Theta_kH(y_k(x)-y)$$ consists of only weak fronts, and
$$U^\delta_s(x,\cdot)=\sum_{k=1}^{N-1}(U_k-U_{k-1})(1-\Theta_k)H(y_k(x)-y)$$ consists of only 3-strong rarefaction fronts.  Let ${p}_a^\delta$ (respectively $p_b^\delta$) be the pressure ahead of the upmost (respectively behind the lowermost) 3-strong rarefaction fronts in $U^\delta$. Then by decreasing of pressure across rarefaction waves (from above to below), and note that pressure increases only passing a 3-shock front in the middle of the rarefaction waves fans (cf. Lemma \ref{l3}), we have
\begin{equation*}\begin{split}
&|\mathrm{TV. }p_s^\delta(x,\cdot)-\mathrm{TV. }p_{\rm b}(x,\cdot)|=  |({p}_a^\delta-{p}_b^\delta)-(p_+-\bar{p})|+ C_2L_3(x)\\
\le &|{p}_a^\delta-{p}_+|+|{p}_b^\delta-\bar{p}|+C_2L_3(x)
\le \mathrm{TV. }p_w^\delta(x,\cdot)+C_2L_3(x)\nonumber\\
\le & \mathrm{TV. }U_w^\delta(x,\cdot)+C_2L_3(x).
\end{split}\end{equation*}
By triangle inequality and \eqref{eq435},
\begin{equation*}\begin{split}
|\mathrm{TV. }p^\delta(x,\cdot)-\mathrm{TV. }p_{\rm b}(x,\cdot)| \le& |\mathrm{TV. }p_w^\delta(x,\cdot)|+|\mathrm{TV. }p_s^\delta(x,\cdot)-\mathrm{TV. }p_{\rm b}(x,\cdot)|\nonumber\\
 \le& 2\mathrm{TV. }U_w^\delta(x,\cdot)+C_2L_3(x)\le CF(x)\le 2C_3C\varepsilon.
\end{split}
\end{equation*}
Particularly this implies that  $\mathrm{TV}. p^\delta(x,\cdot)$ is uniformly bounded. By Lemma \ref{l3}, the total variation of $U^\delta$ introduced by all strong rarefaction fronts is bounded by the total variation  of $p^\delta$, while those introduced  by $1,2,4$-wave fronts could be controlled by $L_0(x)$,  so $\mathrm{TV}. U^\delta(x,\cdot)$ is uniformly bounded (independent of $x$ and $\delta$).

Since $U_0=U_+$, applying Lemma \ref{l4} successively to the middle states $U_1, \cdots, U_k$, we may infer that the corresponding right-hand side of \eqref{E:5.2} is bounded by $C'_1 L_0(x)$, so
\begin{eqnarray}\label{eq438}
U^\delta\in D(U_+,2C'_1C_3\varepsilon)\subset D(U_+,M_1\varepsilon),\end{eqnarray}
where we take $M_1=\max\{2C_3C,2C'_1C_3\}$, and $\varepsilon_0\le\delta_0/M_1$. Thus $U^\delta$ is uniformly bounded (independent of $\delta$).

\begin{remark}\label{rm41}
As shown in \cite[Section 7.5, p.146]{bressan} or \cite[p.530, Section 14.5]{Da}, we can similarly prove that the total variations of $U^\delta$ on space-like curves are uniformly bounded (independent of $\delta$).
\end{remark}

\section{Existence of Entropy Solutions}\label{sec5}

The results in the previous section guarantee that we could construct $(U^\delta(x,y), g^\delta(x))$ for all $0\le x<+\infty.$ Then by standard compactness arguments as shown in \cite[Section IV, A]{CKY2013}, there is a subsequence $\{\delta_j\}_{j=1}^\infty$ that converges to zero, and functions $U, g$ so that $g^{\delta_j}(x)$ converges uniformly to $g(x)$ on any bounded interval, with the estimate \eqref{freeest} following from Remark \ref{rm22}, and $U^{\delta_j}$ converges to $U$ in $C([0,T];L^1)$ after suitable shifts in  $y$-variable. The estimates \eqref{eq:2.1.12} and \eqref{E:1.5} then follow directly from \eqref{eq437}\eqref{eq438}.

Therefore, to complete proof of Theorem \ref{T:1.1}, we need only to show that the limit $(U,g)$ found above is actually an entropy solution to problem \eqref{E:1.2}\eqref{E:1.3}. We follow the idea presented in \cite[pp.299-305]{HR2015}. In fact, once \eqref{eq435} is established, there is little difference from our situation to the standard theory.

Let $\mathcal{NP}(\tau)$ be the set of all non-physical fronts lying on $x=\tau$ in the approximate solution $(U^\delta, g^\delta).$ The key point is to show the total strength of non-physical fronts (called ghost waves in \cite{HR2015}) is bounded by  $O(1)\delta$, rather than $\delta^*$ proved before.

To this end, denote $\mathcal{G}_m(\tau)$ to be the set of  front that lies on $x=\tau$, has generation order $m$, in the approximate solution $(U^\delta, g^\delta).$ Since for $m\ge1$, all fronts of order $m$ are weak ones, we could use (6.32) in \cite[p.301]{HR2015}, (with $n=3$ and $T$ there replaced by $\delta^*$,) to obtain the inequality
\begin{eqnarray}\label{eq51}
\sharp(\mathcal{G}_m(\tau))\le C_1\left(\frac{3\delta^*}{\delta}\right)^{2m-1};
\end{eqnarray}
here $\sharp(A)$ is the cardinal number of a set $A$. For $T_m(\tau)$ being the total strength of fronts in $\mathcal{G}_m(\tau)$, Lemma 6.6 in \cite[p.301]{HR2015} claims that (with $T(t)$ there replaced by $\delta^*$)
\begin{eqnarray}\label{eq52}
T_m(\tau)\le C_1(C_2\delta^*)^m,
\end{eqnarray}
 and $C_1, C_2$ are constants depending only on the background solution.

We also note that once a non-physical front $\epsilon_0$ is produced for the first time at a point $(\tau_0, y_0)$ (i.e., in Cases 1-4, the simplified Riemann solver is adopted at the point), by \eqref{eq36} and all the interaction estimates listed in Lemmas \ref{L:5.1}--\ref{L:4.8}, there must hold
\begin{eqnarray}\label{eq53}
|\epsilon_0|\le C_1\mu_\delta.
\end{eqnarray}
Denote $\epsilon_k$ ($k=1,2,\ldots$) the non-physical front coming from $\epsilon_0$ after it interacts with physical wave fronts $\alpha^1,\cdots, \alpha^k$ for $k$ times (i.e., Cases 5 and 6), then $\epsilon_k$ travels on the same half-line issuing from $(\tau_0, y_0)$ with speed $\hat{\lambda}$, which is space-like.  The interaction estimates show that
\begin{align}\label{eq54}
|\epsilon_k|\le& \epsilon_0(1+C_1|\alpha^1|)(1+C_1|\alpha^2|)\cdots(1+C_1|\alpha^k|)\nonumber\\
\le& C_1\mu_\delta e^{2C_1(|\alpha^1|+\cdots+|\alpha^k|)}\le C_1\mu_\delta e^{C_1'}<C_2\mu_\delta.
\end{align}
Here, for the second inequality, we used that $|\alpha^j|\le\delta^*,$ ($j=1,\ldots, k$,) and for $\delta^*$ small, we have $C_1\delta^*<1/2,$ hence $\ln(1+C_1|\alpha^j|)\le 2C_1|\alpha^j|$ by a simple calculus inequality. The third inequality holds by Remark \ref{rm41}. The constants $C_1, C_1'$ and $C_2$ depend solely on the background solution. So \eqref{eq54} holds for any non-physical front, since $C_2>C_1$.

For a non-physical wave $\epsilon$, denote its generation order to be $\tilde{\epsilon}$. Then
\begin{align}\label{eq55}
T_{\mathcal{NP}}(\tau)\doteq& \sum_{\epsilon\in \mathcal{NP}}|\epsilon|=\sum_{\epsilon: \tilde{\epsilon}< k_0}|\epsilon|+
\sum_{\epsilon: \tilde{\epsilon}\ge k_0}|\epsilon|\nonumber\\
\le&C_2\mu_\delta\sum_{1\le m< k_0}\sharp(\mathcal{G}_m(\tau))+
\sum_{m\ge k_0}T_m(\tau)\nonumber\\
\le&C_2C_1\mu_\delta\sum_{m=1}^{k_0-1}\left(\frac{3\delta^*}{\delta}\right)^{2m-1}
+C_1\frac{(C_2\delta^*)^{k_0}}{1-C_2\delta^*}.
\end{align}
We require $\delta^*$ small so that $C_2\delta^*<1$ as we done before. Now we choose $k_0$ large so that $C_1\frac{(C_2\delta^*)^{k_0}}{1-C_2\delta^*}\le\delta/2.$ Then as $\delta$ and $k_0$ are fixed, we could choose $\mu_\delta$ small so that
$C_2C_1\mu_\delta\sum_{m=1}^{k_0-1}\left(\frac{3\delta^*}{\delta}\right)^{2m-1}\le\delta/2.$

So by using the simplified Riemann solver judiciously, one could obtain that
\begin{eqnarray}\label{eq56}
T_{\mathcal{NP}}(\tau)\le \delta,\qquad \forall \tau>0
\end{eqnarray}
for the approximate solution $(U^\delta, g^\delta)$. Since by our construction, the free-boundary is accurate in each approximate solution, we could proceed in the same way as in \cite[Section 7.4]{bressan} or \cite[pp. 304-305]{HR2015} to show consistency, namely the limit of  $(U^\delta, g^\delta)$ must be an entropy solution to problem \eqref{E:1.2}\eqref{E:1.3}.

As an example, we prove that for every nonnegative test function $\Psi\in C^1_c(\mathbb{R}^2)$, one has
\begin{eqnarray}\label{eq57}
\lim_{\delta\to0}\mathcal{M}^{\delta}\ge0,
\end{eqnarray}
where
\begin{equation}
\mathcal{M}^{\delta}\doteq\iint_{\{x\ge0,y\in\mathbb{R}\}\setminus \mathrm{corner}}\big(\eta(U^{\delta})\Psi_x+q(U^{\delta})
\Psi_y\big)\,\dd x \dd y+\int_{y>0}\eta({U}^{\delta}(0,y))\Psi(0, y)\,\dd y,
\end{equation}
and $\eta(U)=-\rho u S$, $q(U)=-\rho vS$. This justifies the entropy condition for weak solutions.

Recall that any approximate solution $U^\delta$ is piece-wise constant in $\{x>0, y\in\mathbb{R}\}$ (including the static gas), and each piece $\Omega_k$ is separated by $x=0,$ or wave fronts that are straight lines. Let $\Gamma_{\alpha}=\{y=y_{\alpha}(x), x>0\}$ be these lines where $U^\delta$ jumps (including the free-boundary $y=g^\delta(x)$). Then
\begin{equation*}\begin{split}
\mathcal{M}^{\delta}=&\sum_k\iint_{\Omega_k}\big(\eta(U^{\delta})\Psi_x
+q(U^{\delta})\Psi_y\big)\,\dd x\dd y+\int_{y>0}\eta({U^\delta}(0,y))\Psi(0, y)\,\dd y\\
=&\sum_k\int_{\p\Omega_k}\big(\eta(U^{\delta})\Psi, q(U^{\delta})\Psi)\cdot n_k\, \dd s-\sum_k\iint_{\Omega_k}\Big(\eta(U^{\delta})_x+q(U^{\delta})_y\Big)\Psi \,\dd x\dd y\\
&\qquad+\int_{y>0}\eta({U^\delta}(0,y))\Psi(0, y)\,\dd y\\
=&\sum_k\int_{\p\Omega_k}\big(\eta(U^{\delta})\Psi, q(U^{\delta})\Psi)\cdot n_k\, \dd s+\int_{y>0}\eta({U^\delta}(0,y))\Psi(0, y)\,\dd y\\
=&\sum_\alpha\int_{\Gamma_\alpha}\big(\eta(U^{\delta})\Psi, q(U^{\delta})\Psi)\cdot n_\alpha\, \dd s.
\end{split}\end{equation*}
Here $n_k$ is the outer normal vector on $\p\Omega_k$ (the boundary of the polygon $\Omega_k$),  and $n_\alpha$ denotes a normal vector of the line $\Gamma_\alpha$, satisfying
\[
n_\alpha \dd s=\pm(\dot{y}_{\alpha}(x), -1)\,\dd x.
\]

Fix a number  $T>0$ so that the support of $\Psi$ lies in $\{0<x<T\}$, which is independent of $\delta$. Then we get
\[
\mathcal{M}^{\delta}=\sum_\alpha\int^{T}_{0}h_{\delta, \alpha}(x)\Psi(x, y_{\alpha}(x))\,\dd x=\int^{T}_{0}\sum_\alpha h_{\delta, \alpha}(x)\Psi(x, y_{\alpha}(x))\,\dd x,
\]
where
\[
\begin{split}
&h_{\delta, \alpha}(x)=\dot{y}_{\alpha}(x)[\eta(U^{\delta})]-[q(U^{\delta})],\\[3pt]
&[\eta(U^{\delta})]=\eta(U^{\delta})(x, y_{\alpha}(x)+)-\eta(U^{\delta})(x, y_{\alpha}(x)-),\\[3pt]
&[q(U^{\delta})]=q(U^{\delta})(x, y_{\alpha}(x)+)-q(U^{\delta})(x, y_{\alpha}(x)-).
\end{split}
\]

If $\alpha$ is a shock front, then by entropy inequality for Riemann problems, one has $h_{\delta,\alpha}(x)\ge0$.

If $\alpha$ is a characteristic discontinuity, then by Rankine-Hugoniot jump conditions, $h_{\delta,\alpha}(x)\equiv0$.

If $\alpha$ is a rarefaction front, then by our rule of splitting of  rarefaction waves, $\dot{y}_\alpha(x)=\lambda_j(U^\delta(x, y_{\alpha}(x)+))$ for $j=1$ or $3$. Using Taylor expansion as in \cite[p.304]{HR2015}, and note that jump of $U$ across a rarefaction front is bounded by $O(1)\delta$, we have $h_{\delta,\alpha}(x)\le C_1\delta^2$. The number of rarefaction front is bounded by $O(1)/\delta$.

If $\alpha$ is a non-physical front, then $\dot{y}_\alpha(x)=\hat{\lambda}$. But by \eqref{eq56}, $\sum_{\alpha\in\mathcal{NP}(x)}|[U^\delta]|_{y_\alpha(x)}|\le\delta.$

Therefore, by mean value theorem, we have
\begin{eqnarray}
\mathcal{M}^\delta\ge -\int_0^T (O(1)\frac{1}{\delta}C_1\delta^2+C_1\delta)\norm{\Psi}_{L^\infty}\,\dd x= -C_4\delta,
\end{eqnarray}
with $C_4$ depending on $T$, $\norm{\Psi}_{L^\infty}$ and the background solution, as well as $\varepsilon_0$ (the total perturbation of initial data), but not on $\delta$. This proves \eqref{eq57} and finishes the proof of Theorem \ref{T:1.1}.

\section{Appendix}\label{sec6}
We show here that  $k_j>0$ $(j=1, 3)$ in the supersonic domain $\{U:u>c\}$. Set
\[
\theta_{ma}=\arctan \frac{c}{\sqrt{q^2-c^2}}, \quad \theta=\arctan \frac{v}{u}.
\]
Then
\[
\lambda_j=\tan(\theta+(-1)^{\sigma(j)}\theta_{ma})\; \text{ and \;} \sin \theta_{ma}=\frac{c}{q},
\]
where $\sigma(1)=1,$ $\sigma(3)=0.$
Direct calculations yield
\[
\begin{split}
\frac{\partial \lambda_j}{\partial u}=&\frac{\partial \lambda_j}{\partial \theta}\frac{\partial \theta}{\partial u}+\frac{\partial \lambda_j}{\partial \theta_{ma}}
\frac{\partial \theta_{ma}}{\partial u}\\
&=-\frac{\sec^2(\theta+(-1)^{\sigma(j)} \theta_{ma})}{\sqrt{q^2-c^2}}\sin (\theta+(-1)^{\sigma(j)}\theta_{ma}),\\
\frac{\partial \lambda_j}{\partial v}=&\frac{\partial \lambda_j}{\partial \theta}\frac{\partial \theta}{\partial v}+\frac{\partial \lambda_j}{\partial \theta_{ma}}
\frac{\partial \theta_{ma}}{\partial v}\\
&=\frac{\sec^2(\theta+(-1)^{\sigma(j)} \theta_{ma})}{\sqrt{q^2-c^2}}\cos (\theta+(-1)^{\sigma(j)}\theta_{ma}),\\
\frac{\partial \lambda_j}{\partial p}=
&\sec^2(\theta+\theta_{ma})\frac{\partial \theta_{ma}}{\partial p}=(-1)^{\sigma(j)}\sec^2(\theta+(-1)^{\sigma(j)} \theta_{ma})\frac{\gamma}{2\rho c^2}\tan \theta_{ma},\\
\frac{\partial \lambda_j}{\partial \rho}=
&(-1)^{\sigma(j)}\sec^2(\theta+\theta_{ma})\frac{\partial \theta_{ma}}{\partial \rho}=(-1)^{{\sigma(j)}+1}\sec^2(\theta+(-1)^{\sigma(j)} \theta_{ma})\frac{1}{2\rho}\tan \theta_{ma}.
\end{split}
\]
To normalize $r_j$ so that $\nabla \lambda_j \cdot r_j=1$, $j=1, 3$, we need to take
\begin{equation}\label{A1}
k_1(U)=\frac{2\sqrt{q^2-c^2} \cos^3(\theta-\theta_{ma})}{\gamma+1}, \quad k_3(U)=
\frac{2\sqrt{q^2-c^2}\cos^3(\theta+\theta_{ma})}{\gamma+1}.
\end{equation}
Since
\[
u^2(q^2-c^2)-v^2 c^2=(u^2-c^2)q^2>0,
\]
and
\[
\cos(\theta-\theta_{ma})=\frac{vc+u\sqrt{q^2-c^2}}{q^2},\quad \cos(\theta+\theta_{ma})=\frac{u\sqrt{q^2-c^2}-vc}{q^2},
\]
we have $\cos(\theta\pm\theta_{ma})>0$ in the supersonic region $\{u>c\}$. Hence, from \eqref{A1}, we see $k_j>0$ for $j=1, 3$.

\section*{Acknowledgments} The authors are grateful to the Editors and anonymous Referees for their great patience and carefully reading a previous version of the manuscript. Their valuable comments and  suggestions help us substantially revise the manuscript.


\medskip
Received xxxx 20xx; revised xxxx 20xx.
\medskip


\begin{thebibliography}{99}

\bibitem{amadori}
\newblock D. Amadori,
\newblock Initial-boundary value problems for nonlinear systems of conservation laws,
\newblock \emph{NoDEA Nonlinear Differential Equations Appl.}, \textbf{4} (1997),  1--42.

\bibitem{bressan}
\newblock A. Bressan,
\newblock \emph{Hyperbolic Systems of Conservation Laws: The One-Dimensional Cauchy Problem},
\newblock Oxford Lecture Series in Mathematics and its Applications, 20, Oxford University Press, Oxford, 2000.

\bibitem{CKZ1}
\newblock G.-Q. G. Chen,   J. Kuang and Y. Zhang,
\newblock  Two-dimensional steady supersonic exothermically reacting Euler flow past Lipschitz bending walls,
\newblock \emph{ SIAM J. Math. Anal.},  \textbf{49} (2017),  818-873.

\bibitem{CKY1}
\newblock  G.-Q. G. Chen, V. Kukreja and H. Yuan,
\newblock  Stability of transonic characteristic discontinuities in two-dimensional steady compressible Euler flows,
\newblock \emph{J. Math. Phys.},  \textbf{54}  (2013), 021506, 24 pp.

\bibitem{CKY2013}
\newblock  G.-Q. G. Chen, V. Kukreja and H. Yuan,
\newblock  Well-posedness of transonic characteristic discontinuities in two-dimensional steady compressible Euler flows.
\newblock \emph{Z. Angew. Math. Phys.},  \textbf{64}  (2013),  1711--1727.


\bibitem{CZZ-2006}
\newblock  G.-Q. G. Chen, Y. Zhang and D. Zhu,
\newblock  Stability of compressible vortex sheets in steady supersonic Euler flows over Lipschitz walls,
\newblock \emph{SIAM J. Math. Anal.},  \textbf{38} (2006/07), 1660--1693.

\bibitem{CZZ}
\newblock  G.-Q. G. Chen, Y. Zhang and D. Zhu,
\newblock Existence and stability of supersonic Euler flows past Lipschitz wedges,
 \newblock \emph{Arch. Rational Mech. Anal.}, \textbf{181} (2006),  261--330.


\bibitem{CK}
\newblock R. Courant and K. O.  Friedrichs,
\newblock \emph{Supersonic Flow and
Shock Waves},
\newblock Applied Mathematical Sciences, Vol.12,
Wiley-Interscience, New York, 1948.

\bibitem{Da}
\newblock C. M. Dafermos,
\newblock \emph{Hyperbolic Conservation Laws in Continuum Physics}, \newblock 4$^{th}$ edition,  Grundlehren der Mathematischen Wissenschaften [Fundamental Principles of Mathematical Sciences], 325, Springer-Verlag, Berlin, 2016.

\bibitem{D1}
\newblock M. Ding,
\newblock Existence and stability of rarefaction wave to 1-D piston problem for the relativistic full Euler equations,
\newblock \emph{J. Differential Equations}, \textbf{262}  (2017),  6068--6108.

\bibitem{D2}
\newblock M. Ding,
\newblock Stability of rarefaction wave to the 1-D piston problem for exothermically reacting Euler equations,
\newblock \emph{Calc. Var. Partial Differential Equations},  \textbf{56}  (2017),  56:78.

\bibitem{DKZ}
\newblock M. Ding, J. Kuang and Y. Zhang,
\newblock  Global stability of rarefaction wave to the 1-D piston problem for the compressible full Euler equations,
\newblock \emph{J. Math. Anal. Appl.}  \textbf{448}  (2017), 1228--1264.

\bibitem{DL}
\newblock M. Ding and  Y. Li,
\newblock  Stability and non-relativistic limits of rarefaction wave to the 1-D piston problem for the relativistic Euler equations,
\newblock \emph{Z. Angew. Math. Phys.}  \textbf{68}  (2017), Art. 43, 32 pp.


\bibitem{GJ}
\newblock J. Glimm,
\newblock Solutions in the large for nonlinear
hyperbolic systems of equations,
\newblock \emph{Comm. Pure. Appl. Math.}, \textbf{18} (1965),
697--715.

\bibitem{HR2015}
\newblock H. Holden and N. H.  Risebro,
\newblock \emph{Front Tracking for Hyperbolic
Conservation Laws},
\newblock 2$^{nd}$ edition, Applied Mathematical Sciences, 152,
Springer-Verlag, Berlin Heidelberg, 2015.


\bibitem{KYZ}
\newblock V. Kukreja, H. Yuan and Q. Zhao,
\newblock Stability of transonic jet with strong shock in two-dimensional steady compressible Euler flows,
\newblock \emph{J. Differential Equations} \textbf{258} (2015), 2572--2617.


\bibitem{LXY-2016}
\newblock L. Liu, G. Xu and H. Yuan,
\newblock Stability of spherically symmetric subsonic flows and transonic shocks under multidimensional perturbations,
\newblock \emph{Adv. Math.}, \textbf{291} (2016), 696--757.


\bibitem{QX2017}
\newblock A. Qu and W.  Xiang,
\newblock Three-Dimensional Steady Supersonic Euler Flow Past a Concave Cornered Wedge with Lower Pressure at the Downstream,
\newblock \emph{ Arch Rational Mech. Anal.},  (2017). Available from: https://doi.org/10.1007/s00205-017-1197-x


\bibitem{SJ}
\newblock J.  Smoller,
     \newblock \emph{Shock Waves and Reaction-Diffusion Equations},
     \newblock 2$^{nd}$ edition,  Springer-Verlag, New York, 1994.


\bibitem{WY2015}
\newblock Y.-G. Wang and H. Yuan,
\newblock  Weak stability of transonic contact discontinuities in three-dimensional steady non-isentropic compressible Euler flows,
\newblock \emph{Z. Angew. Math. Phys.}, \textbf{66} (2015),  341--388.

\bibitem{WZ}
\newblock Z. Wang and Y. Zhang,
\newblock  Steady supersonic flow past a curved cone,
\newblock \emph{J. Differential Equations}, \textbf{247} (2009), 1817--1850.


\bibitem{Zhang}
\newblock Y. Zhang,
\newblock Steady supersonic flow over a bending wall,
\newblock \emph{Nonlinear Anal. Real World Appl.}, \textbf{12} (2011), 167--189.


\end{thebibliography}
\end{document}